\documentclass[a4paper,twoside,10pt]{article}

\usepackage[english]{babel}
\usepackage[latin1]{inputenc}
\usepackage{lmodern}
\usepackage[T1]{fontenc}
\usepackage{indentfirst}
\usepackage{amsmath,amssymb}

\usepackage[a4paper,top=3cm,bottom=3cm,left=3cm,right=3cm,%
bindingoffset=5mm]{geometry}
\usepackage{lmodern}
\usepackage[T1]{fontenc}
\usepackage{lettrine}
\usepackage{booktabs}

\usepackage{authblk}
\usepackage{emp}
\usepackage{amsthm}
\usepackage{amsmath,amssymb,amsthm}
\usepackage{braket}
\usepackage{chngpage}
\usepackage{cases}
\usepackage{graphicx}
\usepackage{tabularx}
\usepackage{multirow}

\newcommand{\numberset}{\mathbb}

\newcommand{\R}{\numberset{R}}

\newcommand{\B}{\numberset{B}}
\newcommand{\Pk}{\numberset{P}}
\renewcommand{\epsilon}{\varepsilon}
\renewcommand{\theta}{\vartheta}
\renewcommand{\rho}{\varrho}
\renewcommand{\phi}{\varphi}

{\left\lbrace\begin{array}{@{}l@{}}}%
{\end{array}\right.}
\theoremstyle{definition}

\theoremstyle{remark}
\newtheorem{remark}{Remark}[section]

\theoremstyle{remark}
\newtheorem{test}{Test}[section]

\theoremstyle{plain}
\newtheorem{theorem}{Theorem}[section]
\newtheorem{proposition}{Proposition}[section]

\newtheorem{lemma}{Lemma}[section]

\usepackage{lipsum}

\usepackage{setspace}

\author{Giuseppe Vacca  
  \thanks{Dipartimento di Matematica,  Universit\`a degli Studi di Bari, Via Edoardo Orabona, 4 - 70125 Bari, E-mail: giuseppe.vacca@uniba.it.} 
  }

\title{\textbf{Virtual Element Methods for hyperbolic problems on polygonal meshes}}

\date{\today}

\begin{document}

\maketitle
\begin{abstract}
In the present paper we develop  the Virtual Element Method for hyperbolic problems on polygonal meshes, considering the linear wave equations as our model problem. After presenting the semi-discrete scheme, we derive the convergence estimates in $H^1$ semi-norm and $L^2$ norm.  Moreover we develop a theoretical analysis on the stability for the fully discrete problem by comparing the Newmark method and the Bathe method. Finally we show the practical behaviour of the proposed method through a large array of numerical tests.
\end{abstract}

\section{Introduction}
\label{sec:1}

The \textbf{Virtual Element Methods} (in short, VEM or VEMs) is a very recent technique for solving partial differential equations. VEMs were lately introduced in \cite{VEM-volley} as a generalization of the finite element method on polyhedral or polygonal meshes. 

The \textbf{virtual element spaces} are similar to the usual polynomial spaces with the addition of suitable (and unknown!) non-polynomial functions. The main idea behind VEM is to define \textbf{approximated discrete bilinear forms} that are \textbf{computable} only using the \textbf{degrees of freedom}. The key of the method is to define suitable \textbf{projections} (for instance gradient projection or $L^2$ projection) onto the space of polynomials that are computable on the basis of the degrees of freedom.  Using  these projections, the bilinear forms (e.g. the stiffness matrix, the mass matrix and so on)  require only integration of polynomials on the (polytopal) element in order to be computed. Moreover, the ensuing discrete solution is conforming and the accuracy granted by such discrete bilinear forms turns out to be sufficient to recover the correct order of convergence.
Following such approach, VEM is able to make use of very general polygonal/polyhedral meshes without the need to integrate complex non-polynomial functions on the elements (as polygonal FEM do) and without loss of accuracy. As a consequence, VEM is not restricted to low order converge and can be easily applied to three dimensions and use non convex (even non simply connected) elements.

An additional peculiarity of the VEMs is the satisfaction of the \textbf{patch test} used
by engineers for testing the quality of the methods. Roughly speaking, a method satisfies the patch test if it is able to give the exact solution whenever this is a global polynomial of the selected degree of accuracy.

In \cite{VEM-enhanced} the authors introduce a variant of the virtual element method presented in \cite{VEM-volley} that allows to compute the exact $L^2$ projection of the virtual space onto the space of polynomials and extends the VEMs technology to the three-dimensional case. A helpful paper for the computer implementation of the method is \cite{VEM-hitchhikers}. In \cite{VEM-conforming} the authors construct Virtual Element spaces that are $H({\rm div})$-conforming and $H({\rm curl})$-conforming. 

The Virtual Element Method has been developed successfully for a large range of problems:  the linear elasticity problems, both for the compressible and the nearly incompressible case \cite{VEM-elasticity, paulinopost}, a stream formulation of VEMs for the Stokes problem \cite{VEM-stream}, the non-linear elastic and inelastic deformation problems, mainly focusing on a small deformation regime \cite{BLM15}, the Darcy problem in mixed form \cite{VEM-mixed}, the plate bending problem \cite{Brezzi:Marini:plates}, the Steklov eigenvalue problem \cite{mora2015virtual}, the general second order elliptic problems in primal \cite{VEM-general} and mixed form \cite{VEM-mixedgeneral}, the Cahn-Hilliard equation \cite{VEM-cahn}, the Helmholtz problem \cite{VEM-helmholtz},  the discrete fracture network simulations \cite{berrone}, the time-dependent diffusion problems \cite{vaccabeirao} and the Stokes problem \cite{VEM-preprint}.  In \cite{VEM-nonconforming, VEM-nonconforming1} the authors present a non-conforming Virtual Element Space. 
Finally in \cite{VEM-serendipity} the authors introduce the last version of Virtual Element spaces, the Serendipity VEM spaces that, in analogy with the Serendipity FEMs, allows to reduce the number of degrees of freedom.

Recently in \cite{lopezvacca, mfdpreprint}, Mimetic Finite Difference methods \cite{BLM11book} (technique having common features with VEM) have been applied to the space discretization of PDEs of parabolic and hyperbolic type in two dimension, showing how this technique preserves invariants of the solution better than classical space discretizations such as finite difference methods. In the present contribution we develop  the Virtual Element Method for hyperbolic problems. We consider as a model problem the classical time-dependent wave equations. The discretisation of the problem requires the introduction of two discrete bilinear forms, one being the approximated grad-grad form of the stationary case \cite{VEM-volley} and the other being a discrete counterpart of the $L^2$ scalar product. The latter is built making use of the enhancements techniques of \cite{VEM-enhanced}. In the paper we focus our attention on the bi-dimensional case and we develop a full theoretical analysis, first analysing on the error between the semi-discrete and the continuous problems and later giving two examples of fully discrete problems. Finally, a large range of numerical tests in accordance with the theoretical derivations is presented. 


The paper is organized as follows. In Section \ref{sec:2} we introduce the model continuous problem. 
In Section \ref{sec:3} we present its VEM discretisation and the analysis of the error for the semi-discrete problem. 
In Section \ref{sec:4} we detail the theoretical features of the fully discrete scheme, in particular we analyse the convergence and the stability properties for the fully discrete problem by using the Newmark method and the Bathe method as time integrator method.  
Finally, in Section \ref{sec:5} we show the numerical tests.

\section{The continuous problems}
\label{sec:2}

We consider the second order evolution problems in time, in particular we study the wave equations as model hyperbolic problem.
Let  $\Omega \subset \R^2$ be the polygonal domain of interest. Then the mathematical problem is given by:
\begin{equation}
\label{eq:hyper primale}
\left\{
\begin{aligned}
& u_{tt} - \Delta u = f \qquad  \qquad & \text{ in $\Omega \times (0,T)$,} \\
& u = 0 \qquad \qquad & \text{on $\partial \Omega \times (0,T)$,} \\
& u(\cdot, 0) = u_0, \qquad   u_t(\cdot, 0) = z_0 \qquad & \text{ in $\Omega$,} 
\end{aligned}
\right.
\end{equation}
where $u$ represents the unknown variable of interest, $u_t$ and $u_{tt}$  denote respectively its first  and  second order time derivative. We assume the external force $f \in L^2(\Omega \times (0,T))$ and the initial data $u_0$, $z_0 \in H^1_0(\Omega)$.
Then a standard variational formulation of Problem \eqref{eq:hyper primale} is:
\begin{equation}
\label{eq:hyper var}
\left \{
\begin{aligned}
& \text{find $u \in C^0(0, T; \, H^1_0(\Omega)) \cap C^1(0,T; \, L^2(\Omega))
$, such that} \\
& ( u_{tt}(t), \, v ) + a(u(t), \, v) =  \langle f(t), \, v \rangle \quad \text{for all $v \in H^1_0(\Omega)$, for a.e. $t$  in  $(0,T)$} \\
& u(0) = u_0, \qquad u_t(0) = z_0,
\end{aligned}
\right.
\end{equation}
where 
\begin{itemize}
\item the derivative $u_{tt}$ above is to be intended in the weak sense in $(0, T)$,
\item $( \cdot \, , \, \cdot) \colon L^2(\Omega) \times L^2(\Omega) \to \R$ denotes the standard $L^2$ scalar product on $\Omega$,
\item $a( \cdot \, , \, \cdot) \colon H^1_0(\Omega) \times H^1_0(\Omega) \to \R$ denotes the grad-grad form $a(u, \, v) = (\nabla u, \, \nabla v)$,
\item $\langle f(t), \,  \cdot \rangle \colon H^1_0(\Omega) \to \R$ denotes the duality product in $H_0^1(\Omega)$.
\end{itemize}

It is well known (see for instance \cite{raviart}) that the bilinear form $a( \cdot \, , \, \cdot)$ is continuous and coercive, i.e. there exist two uniform positive constant $a$ and $\alpha$ such that
\[
a( u , \, v) \leq a \|u\|_{H^1(\Omega)} \|v\|_{H^1(\Omega)} \qquad  a( v , \, v) \geq \alpha \|v\|^2_{H^1(\Omega)} \qquad \text{for all $u, \, v \in H^1_0(\Omega)$,}
\]
then Problem \eqref{eq:hyper var} has a unique solution $u(t)$ such that
\[
\left( a(u(t), \, u(t)) + \|u_t(t)\|^2_{L^2(\Omega)} \right)^{\frac{1}{2}} \leq \left( a(u_0, \, u_0) + \|z_0\|^2_{L^2(\Omega)} \right)^{\frac{1}{2}} + |f|_{L^1(0, t, \, L^2(\Omega))} \quad \text{$\forall t \in (0, T)$.}
\]

In the rest of the paper we will make use of the following notation. 
We will indicate the classical Sobolev semi-norms (and analogously for the norms) with the shorter symbols 
\[
|v|_{s}=|v|_{H^s(\Omega)} \ , \qquad |v|_{s,\omega}=|v|_{H^s(\omega)}
\]
for any non-negative constant $s\in\mathbb{R}$, open subset $\omega \subseteq \Omega$ and for all $v\in H^s(\Omega)$, 
while $C$ will denote a generic positive constant independent of the mesh diameter $h$ and time step size $\tau$ and that may change at each occurrence.

\section{Virtual formulation of the wave equations}
\label{sec:3}

We here outline the Virtual Element discretization of problem \eqref{eq:hyper var}.  We will make use of various tools from the Virtual Element technology, that will be described briefly; we  
refer the interested reader to the papers \cite{VEM-volley,VEM-enhanced,VEM-hitchhikers}) for a deeper presentation. Finally we derive the convergence estimates in $H^1$ semi-norm and $L^2$ norm.

\subsection{Virtual spaces and bilinear forms}
\label{sub:3.1}

In the outline below we focus on the bi-dimensional case $d=2 $, the three-dimensional one being analogous but more technical.
We start by introducing the Virtual Element space used in the Galerkin-like discretisation of problem \eqref{eq:hyper var}; this needs a few steps.

Let $\mathcal{T}_h$ be an unstructured mesh of $\Omega$ into nonoverlapping polygons with flat faces, where
\[
h_E :=  {\rm diameter}(E) \qquad h := \sup_{E \in \mathcal{T}_h} h_E.
\]
In the following we take on the element $E \in \mathcal{T}_h$ the regularity assumptions listed, for instance, in \cite{VEM-volley}. We require that for all $h$, each element $E \in \mathcal{T}_h$ fulfils the following assumptions:
\begin{itemize}
\item $\boldsymbol{A1}$: $E$ is a simply-connected  polygon with boundary made of a finite number of straight line segments,
\item $\boldsymbol{A2}$: $E$ is star-shaped with respect to a ball of radius greater than $\gamma \, h_E$,  
\item $\boldsymbol{A3}$: the distance between any two verteces of $E$ is greater that $c \, h_E$,
\end{itemize}
where $\gamma$ and $c$ are positive constant independent by $h$ and $E$.

Under the assumptions $\boldsymbol{A1}$, $\boldsymbol{A2}$, $\boldsymbol{A3}$, according with the classical \textbf{Scott-Dupont} theory (see \cite{scott}) we have the following fundamental approximation result.

\begin{theorem}
\label{thm:scott}
Let $E \in \mathcal{T}_h$ and $k \geq 1$, then for all $u \in H^{s+1}(E)$ with $0 \leq s \leq k$, there exists a  polynomial function $u_{\pi}$ on $E$ of degree less or equal than $k$,  such that
\begin{equation}
\label{eq:scottpre}
\|u- u_{\pi}\|_{0,\, E} + h_E \, \|u - u_{\pi} \|_{1, \,E} \leq C \, h_E^{s+1} \, | u|_{s+1, \, E}.
\end{equation}
\end{theorem}
Let $k \in \mathbb{N}$, $k \geq 1$, represent the \textbf{polynomial degree} of the method and let us introduce the following useful notations, for all $E \in \set{\mathcal{T}_h}_h$:
\begin{itemize}
\item $\Pk_k(E)$ the set of polynomials on $E$ of degree $\leq k$,
\item $\B_k(\partial E) := \{v \in C^0(\partial E) \quad \text{s.t} \quad v_{|_e} \in \Pk_k(e) \quad \text{for all edge  $e \subset \partial E$} \, \}$, 
\end{itemize}
where we use the convention $\Pk_{-1}(E) = \{ 0\}$. 
Let us introduce the local counterparts of the bilinear form $a(\cdot,\cdot)$:
\[
a(u,\, v) =: \sum_{E \in \mathcal{T}_h}a^E(u, \, v) \qquad \text{for all $u$,$v \in V$.}
\]
Let now $\Pi ^{\nabla,E}_k \colon H^1(E) \to \Pk_k(E)$ be the \textbf{energy projection operator} (i.e. the $H^1$- seminorm projector) defined by
\[
\left\{
\begin{aligned}
& a^E(q_k, \, v - \, {\Pi}_{k}^{\nabla,E}   v) = 0 \qquad  \text{for all $q_k \in \Pk_k(E)$,} \\
& P^{0,E}(v - \,  {\Pi}_{k}^{\nabla,E}  v) = 0 \, ,
\end{aligned}
\right.
\]
where $P^{0,E} \colon H^1(E) \to \R$ can be taken as
\[
\begin{split}
P^{0,E}(v) &:= \frac{1}{|\partial E|} \int_{\partial E} v\, {\rm d}s \qquad \text{for $k=1$}, \\
P^{0,E}(v) &:= \frac{1}{|E|} \int_E v \, {\rm d}x \qquad \text{for $k >1$}.
\end{split}
\]
Moreover let us denote with $\Pi^{0,E}_{k} \colon L^2(E) \to \Pk_{k}(E)$, the $\boldsymbol{L^2(E)}$ \textbf{projection operator} onto the space $\Pk_k(E)$, i.e.
\[
(q_k, \, v - \Pi^{0,E}_{k} v)_{L^2(E)} = 0 \qquad \text{for all $q_k \in \Pk_k(E)$.}
\]
It is clear that $\Pi^{\nabla, E}$ and $\Pi^{0, E}$ correspond to the identity operator on the space $\Pk_k(E)$.

For all $E \in\mathcal{T}_h$, the \textbf{augmented virtual local space} $\widehat{V}_h^E$ is defined by
\[
\widehat{V}_h^E = \left\{ v \in H^1(E)  \quad \text{s.t.} \quad  v \in \B_k(\partial E), \, \Delta v \in \Pk_{k}(E) \right\}.
\]
Now we define the \textbf{enhanced Virtual Element space}, the restriction $W_h^E$ of $\widehat{V}_h^E$ given by
\begin{equation}
\label{eq:localspace}
W_h^E := \left\{ w \in \widehat{V}_h^E \quad \text{s.t.} \quad   \left(w - \Pi^{\nabla,E}_k w, \, q \right)_{L^2(E)} \quad \text{for all $q \in \Pk_{k}(E)/\Pk_{k-2} (E)$} \right\} ,
\end{equation}
where the symbol $\Pk_{k}(E)/\Pk_{k-2} (E)$ denotes the polynomials of degree $k$ living on $E$ that are $L^2-$orthogonal to all polynomials of degree $k-2$ on $E$.
We notice that, in general, for $v \in H^1(E)$  it holds that 
\[
\int_E (\Pi^{\nabla,E}_k v) \, q \, {\rm d}E \neq \int_E v \, q \, {\rm d}E \qquad \text{for $q \in\Pk_{k}(E)/\Pk_{k-2} (E)$}.
\]
Moreover we introduce  the following set $\mathbf{D}$ of linear operators from $W_h^E$ into $\R$. For all $v \in W_h^E$ we take (see Figure \ref{fig:dofsloc}):
\begin{itemize}
\item $\mathbf{D1}$: the values of $v$ at the $n_E$ vertexes of the polygon $E$,
\item $\mathbf{D2}$:  the values of $v$ at $k-1$ distinct points of every edge $e \in \partial E$ (for example we can take the $k-1$ internal points of the $(k+1)$-Gauss-Lobatto quadrature rule  in $e$, as suggested in \cite{VEM-hitchhikers}),
\item $\mathbf{D3}$: the moments up to order $k-2$  of $v$ in $E$, i.e.
\[
\int_E v \, q_{k-2} \,{\rm d}x \qquad \text{for all $q_{k-2} \in \Pk_{k-2}(E)$.}
\]
\end{itemize}
\begin{figure}[!h]
\center{
\includegraphics[scale=0.15, draft=false]{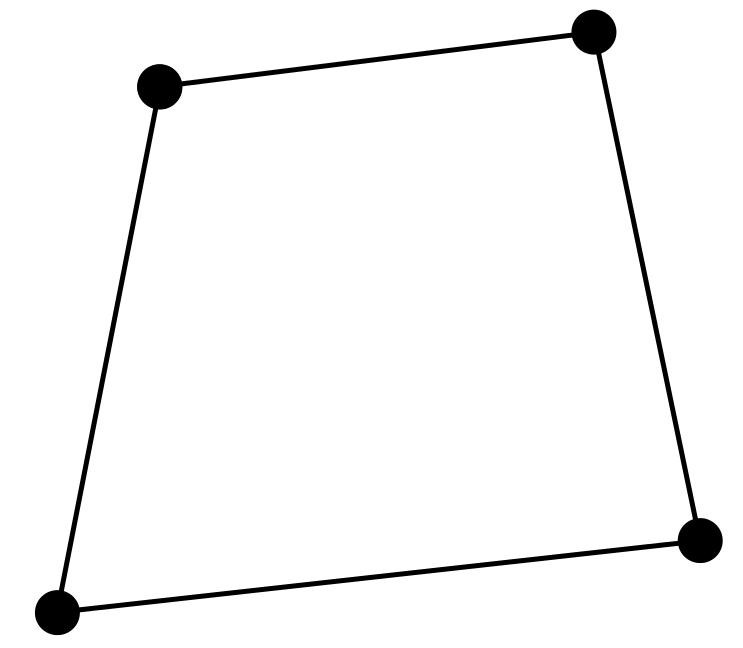} \qquad
\includegraphics[scale=0.15, draft=false]{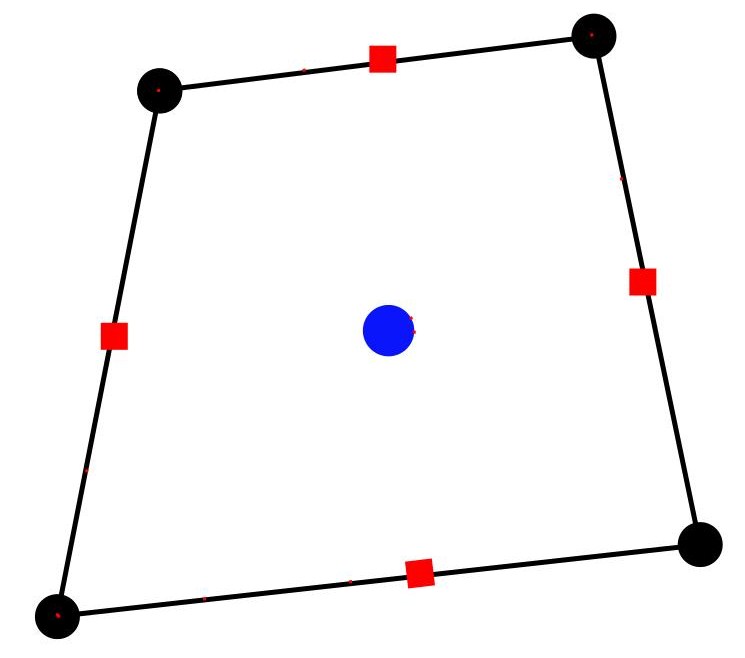} \qquad
\includegraphics[scale=0.15, draft=false]{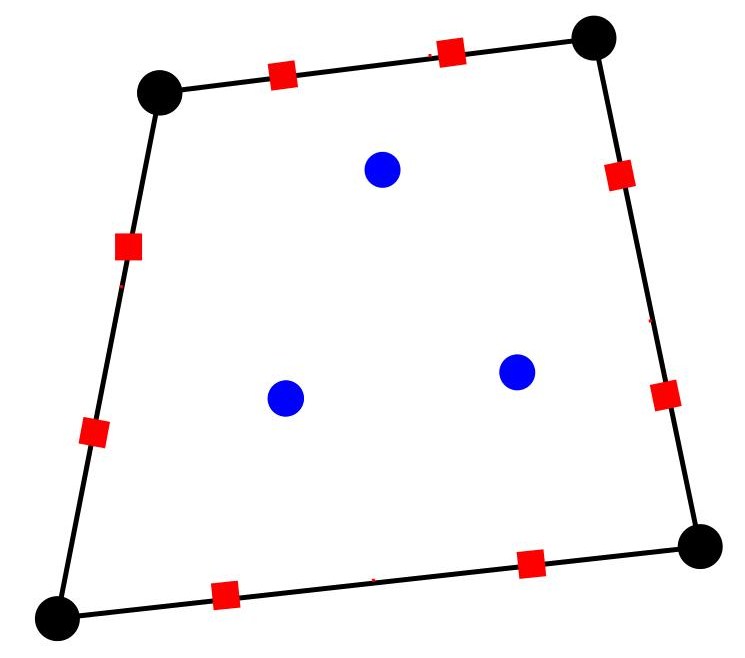}
\caption{Degrees of freedom for $k=1, 2, 3$. We denote $\mathbf{D1}$ with the black dots, $\mathbf{D2}$ with the red squares, $\mathbf{D3}$ with the blue dots.}
\label{fig:dofsloc}
}
\end{figure}

The enhanced space $W_h^E$ has three fundamental properties (see \cite{VEM-enhanced} for a proof): 
\begin{itemize}
\item $\Pk_k(E) \subseteq W_h^E$, that guarantees the good approximation properties for the space,
\item the set of linear operators $\mathbf{D1}$, $\mathbf{D2}$, $\mathbf{D3}$ constitutes a set of \textbf{degrees of freedom} for the space $W_h^E$, 
\item the energy projection operator and the $L^2$-projection operator on the space $W_h^E$
\[
\Pi^{\nabla, E}_{k} \colon W_h^E \to \Pk_k(E), \qquad  \Pi^{0, E}_{k} \colon W_h^E \to \Pk_k(E)
\] 
are \textbf{exactly computable} (only) on the basis of the degrees of freedom.
\end{itemize}
Therefore we have that
\[
{\rm dim} \left(  W_h^E \right) = n_E \, k + \frac{k(k-1)}{2},    
\]
where $n_E$ is the number of vertexes of the polygon $E$.

\begin{remark}
\label{oss:dofs}
We note that the operators $\mathbf{D1}$, $\mathbf{D2}$ are sufficient to uniquely define $v$ on the boundary of $E$. The degrees of freedom $\mathbf{D3}$ allow to compute the $L^2$ projection of the virtual space $W_h^E$ onto $\Pk_{k-2}(E)$. The condition in the definition  \eqref{eq:localspace} allows us to compute the $L^2$ projection onto $\Pk_{k}(E)$, by using the energy projection $\Pi^{\nabla, E}$.
\end{remark}

We have therefore introduced a set of local spaces $W_h^E$ and, thanks to the  properties listed here above, the associated local degrees of freedom. The global discrete space can now be assembled in the classical finite element fashion, yielding
\begin{equation}
\label{eq:globalspace}
W_h = \left\{ w \in H^1_0(\Omega) \quad \text{s.t} \quad w_{|_E} \in W_h^E  \quad \text{for all $E \in \mathcal{T}_h$} \right\} 
\end{equation}
and it holds that
\[
{\rm dim}(W_h) =   n_V + (k-1) n_e + n_P \, \frac{(k-1)(k-2)}{2}  
\]
where $n_P$ (resp. $n_e$ and $n_V$) is the number of elements (resp. internal edges and vertexes) in $\mathcal{T}_h$, and the following constitute the \textbf{global degrees of freedom} for all $v \in W_h$ (see Figure \ref{fig:dofsglob}):
\begin{itemize}
\item $\mathbf{GD1}$: the values of $v$ at the $n_V$ internal verteces,
\item $\mathbf{GD2}$:  the values of $v$ at $k-1$ distinct points of every internal edge $e$,
\item $\mathbf{GD3}$: the moments up to order $k-2$  of $v$ on each element $E \in \mathcal{T}_h$, i.e.
\[
\int_E v \, q_{k-2} \,{\rm d}x \qquad \text{for all $q_{k-2} \in \Pk_{k-2}(E)$.}
\]
\end{itemize}
\begin{figure}[!h]
\center{
\includegraphics[scale=0.15, draft=false]{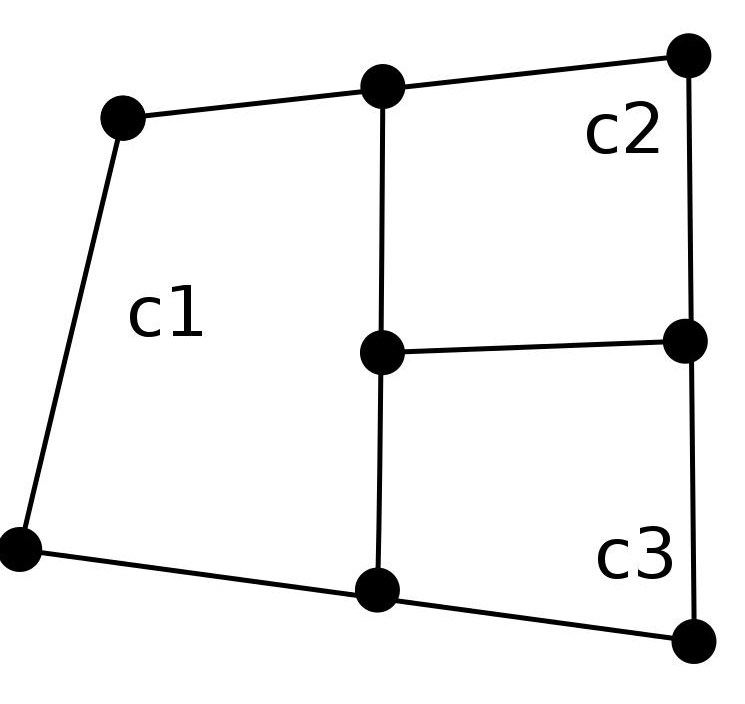} \qquad
\includegraphics[scale=0.15, draft=false]{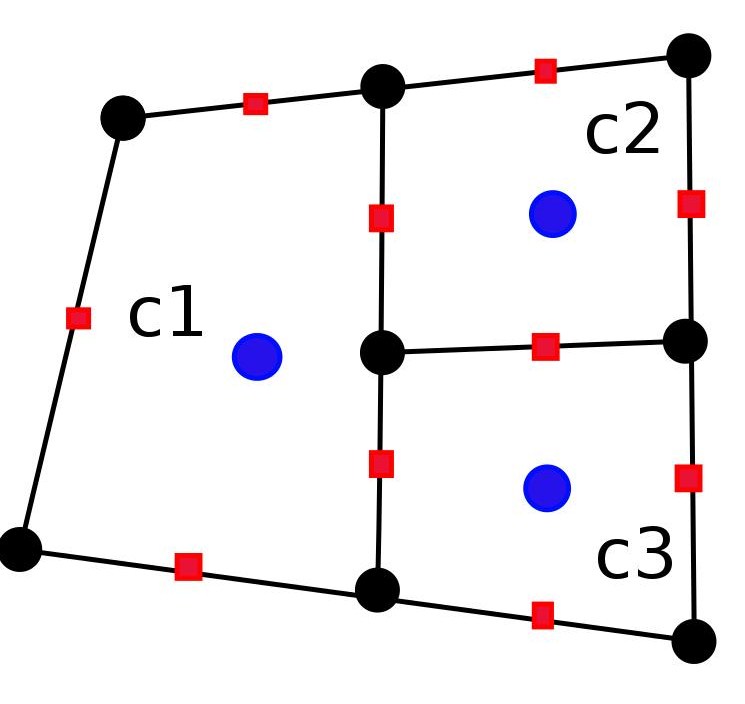} \qquad
\includegraphics[scale=0.15, draft=false]{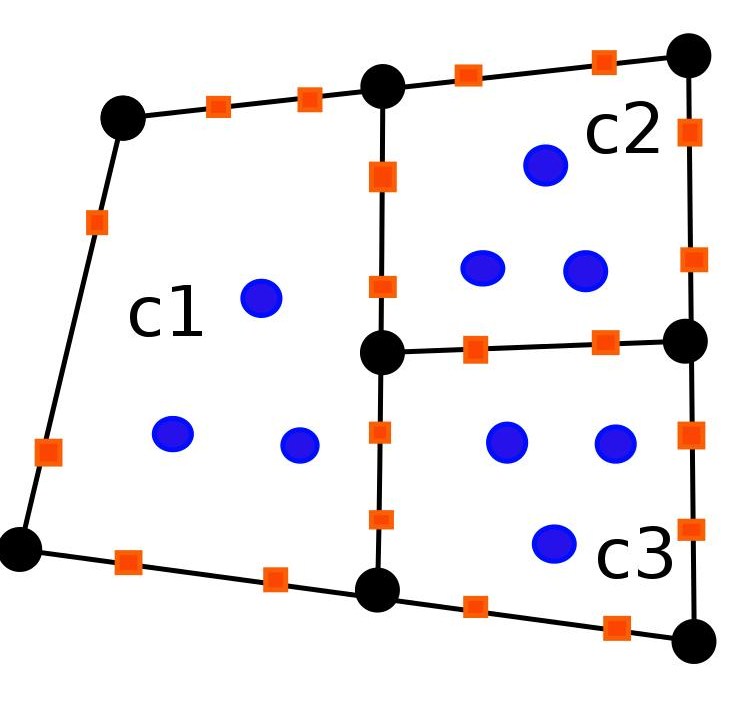}
\caption{Degrees of freedom for $k=1, 2, 3$. We denote $\mathbf{GD1}$ with the black dots, $\mathbf{GD2}$ with the red squares, $\mathbf{GD3}$ with the blue dots.}
\label{fig:dofsglob}
}
\end{figure}

The following useful approximation results hold \cite{mora2015virtual}:
\begin{proposition}
\label{prp:mora}
Let $u \in H^1_0(\Omega) \cap H^{s+1}(\Omega)$ with $0 \leq s \leq k$. Under the mesh assumptions $\mathbf{A1}$, $\mathbf{A2}$, $\mathbf{A3}$
on the decomposition $\mathcal{T}_h$, there exists $u_I \in W_h$  such that
\begin{equation}
\label{eq:interpolation}
\|u - u_I\|_{0} + h \, | u - u_I |_{1} \leq C \,h^{s+1} \, |u|_{s+1}
\end{equation}
where $C$ is a constant independent of $h$.
\end{proposition}

The crucial observation is that   for all $q_k \in \Pk_k(E)$ and for all $v_h \in V_h^E$, the quantities  $a^E(q_k, \, v_h)$ and $(q_k, \, v_h)_{L^2(E)}$ are \textbf{exactly computable} on the basis of degrees of freedom $\textbf{D}$. 
However, for an arbitrary pair $(u, \,v) \in W_h^E$ the quantities $a^E(u, \, v)$ and $(u, \, v)_{L^2(E)}$ are not computable.  
We now define a computable discrete \textbf{virtual local bilinear forms}
\[
a_h^E(\cdot, \,\cdot) \colon W_h^E \times W_h^E \to \R \qquad m_h^E(\cdot, \,\cdot) \colon W_h^E \times W_h^E
\]
approximating the continuous form $a^E(\cdot, \, \cdot)$ and $(\cdot, \, \cdot)_{L^2(E)}$, in the sense that the following properties are satisfied:
\begin{itemize}
\item $\mathbf{k}$\textbf{-consistency}: for all $q_k \in \Pk_k(E)$ and $v_h \in W_h^E$
\[
a_h^E(q_k, \, v_h) = a^E( q_k, \, v_h), \qquad m_h^E(q_k, \, v_h) = ( q_k, \, v_h)_{L^2(E)}
\]
\item \textbf{stability}:  there exist  positive constants $\alpha_*$, $\alpha^*$ and $\beta_*$, $\beta^*$, independent of $h$ and $E$, such that, for all $v_h \in W_h^E$, it holds
\begin{gather}
\label{eq:stabilitygrad}
\alpha_*\, a^E(v_h, \, v_h) \leq a_h^E(v_h, \, v_h) \leq \alpha^*\, a^E(v_h, \, v_h) \\
\beta_*\, (v_h, \, v_h)_{L^2(E)} \leq m_h^E(v_h, \, v_h) \leq \beta^*\, (v_h, \, v_h)_{L^2(E)} .
\end{gather}
\end{itemize}
Following the VEM framework, we can set
\begin{gather}
\label{eq:a_hlocalform}
a_h^E(u_h, \, v_h) := a^E \left({\Pi}_{k}^{\nabla,E} u_h, \,{\Pi}_{k}^{\nabla,E} v_h \right) + \mathcal{S}^E \left((I -{\Pi}_{k}^{\nabla,E}) u_h, \,(I -{\Pi}_{k}^{\nabla,E}) v_h \right) \\
\label{eq:m_hlocalform}
m_h^E(u_h, \, v_h) :=  \left({\Pi}_{k}^{0, E} u_h, \,{\Pi}_{k}^{0, E} v_h \right)_{L^2(E)} + \mathcal{R}^E \left((I -{\Pi}_{k}^{0, E}) u_h, \,(I -{\Pi}_{k}^{0, E}) v_h \right)
\end{gather}
for all $u_h$, $v_h \in W_h^E$, where we have introduced  (symmetric) stabilizing bilinear forms 
\[
\mathcal{S}^E \colon W_h^E \times W_h^E \to \R, \qquad \mathcal{R}^E \colon W_h^E \times W_h^E \to \R
\]
that satisfies 
\[
c_* a^E(v_h, \, v_h) \leq  \mathcal{S}^E(v_h, \, v_h) \leq c^* a^E(v_h, \, v_h) 
\]
for all $v_h \in W_h^E$ such that ${\Pi}_{k}^{\nabla,E} v_h= 0$, and
\[
c_* (v_h, \, v_h)_{L^2(E)} \leq  \mathcal{R}^E(v_h, \, v_h) \leq c^*(v_h, \, v_h)_{L^2(E)} 
\]
for all $v_h \in W_h^E$ such that ${\Pi}_{k}^{0, E} v_h= 0$, for two positive constants $c_*$ and $c^*$.

We define the global approximated bilinear forms $a_h(\cdot, \, \cdot) \colon W_h \times W_h \to \R$ and $m_h(\cdot, \, \cdot) \colon W_h \times W_h \to \R$ by simply summing the local contributions:
\begin{equation}
\label{eq:globalforms}
a_h(u_h, \, v_h) := \sum_{E \in \mathcal{T}_h}  a_h^E(u_h, \, v_h)  \qquad m_h(u_h, \, v_h) := \sum_{E \in \mathcal{T}_h}  m_h^E(u_h, \, v_h)  \qquad \text{for all $u_h$, $v_h \in W_h$.}
\end{equation}
We notice that the symmetry of $a_h$ and $m_h$ and the stability conditions stated before imply the continuity of the bilinear forms, i.e.
\begin{equation}
\label{eq:cont}
a_h(u, \, v) \leq \alpha^* |u|_{1} \, |v|_{1} \qquad m_h(u, \, v) \leq \beta^* \|u\|_{0} \, \|v\|_{0} \qquad \text{for all $u$, $v \in W_h$.}
\end{equation}
Finally we introduce the approximated $H^1$ semi-norm and the approximated $L^2$ norm defined by
\begin{equation}
\label{eq:normah}
|v|_{1,h}^2 := a_h(v, \, v), \qquad  \|v\|_{0,h}^2 := m_h(v, \, v) \qquad \text{for all $v \in W_h$.}
\end{equation}

For the definition of the method we have to construct a \textbf{computable approximation of the right-hand side} $\langle f_h(t), \, v_h \rangle$.  We define the approximated load term $f_h(t)$ for all $t \in (0, T)$ as 
\begin{equation}
\label{eq:right}
f_h(t) := \Pi_{k}^{0,E} f(t) \qquad \text{for all $E \in \mathcal{T}_h$}
\end{equation}
that is is computable directly from degrees of freedom.

The last step in the construction of the virtual method is the definition of suitable discrete initial data. More precisely, we set $u_{h,0}$ (resp. $z_{h,0}$) as the ``interpolant'' of $u_0$ (resp. $z_0$) in $W_h$ by imposing  
\begin{equation}
\label{eq:uh0}
\mathbf{D}(u_{h,0}) = \mathbf{D}(u_0), \qquad \mathbf{D}(z_{h,0}) = \mathbf{D}(z_0).
\end{equation}

\subsection{Virtual semi-discrete problems}
\label{sub:3.2}

We are now ready to state the proposed semi-discrete problem. Referring to~\eqref{eq:globalforms}, \eqref{eq:right} and \eqref{eq:uh0}, we consider the virtual element problem:
\begin{equation}
\label{eq:hyper vir}
\left \{
\begin{aligned}
& \text{find $u_h \in C^0(0, T; \, W_h)) \cap C^1(0,T; \, W_h)
$, such that} \\
& m_h( u_{h, tt}(t), \, v_h ) + a_h(u_h(t), \, v_h) =  \langle f_h(t), \, v_h \rangle \quad  \text{for all $v_h \in W_h$,  for a.e. $t$  in $(0,T)$} \\
& u_h(0) = u_{h, 0}, \qquad u_{h,t}(0) = z_{h,0}.
\end{aligned}
\right.
\end{equation} 

Let $Ndof$ the number of degrees of freedom of the problem, and let us observe that, for the symmetry and the stability conditions of the bilinear forms $a_h$ and $m_h$, there exist 
\[
0 <\lambda_h^{(1)} \leq \dots\leq \lambda_h^{(Ndof)}
\]
and $\set{w_h^{(n)}}_{1, \dots, Ndof}$ orthonormal basis of $W_h$ with respect to $m_h(\cdot, \cdot)$, such that
\begin{equation}
\label{eq:eigenvalue}
a_h(w_h^{(n)}, v_h) = \lambda_h^{(n)} m_h(w_h^{(n)}, v_h) \qquad \text{for all $v_h \in W_h$, for $n=1, \dots, Ndof$.}
\end{equation}
Let $\mu_h^{(n)} := \sqrt{ \lambda_h^{(n)}}$. With this notation we can state the following theorem.
\begin{theorem}
\label{esistenza}
Problem \eqref{eq:hyper vir} has a unique solution, given by
\begin{multline}
\label{soluzione}
u_h(t) := \sum_{n=1}^{Ndof} \left(
m_h(u_{h,0}, \, w_h^{(n)}) \, \cos (\mu_h^{(n)} t)  + \frac{1}{\mu_h^{(n)} } m_h(z_{h,0}, \, w_h^{(n)}) \, \sin (\mu_h^{(n)} t)  + \right. \\  
\left .  + \frac{1}{\mu_h^{(n)} } \, \int_0^t \langle f_h(s), \, w_h^{(n)} \rangle \,  \sin (\mu_h^{(n)} (t-s))  \, {\rm d}s \right
)w_h^{(n)}.
\end{multline} 
Moreover it holds that
\[
\left( a_h(u_h(t), \, u_h(t)) + \|u_{h,t}(t)\|_{h,0}^2 \right)^{\frac{1}{2}} \leq \left( a_h(u_{h, 0}, \, u_{h, 0}) + \|z_{h, 0}\|_{h, 0} \right)^{\frac{1}{2}} + |f_h|_{L^1(0, t, L^2(\Omega))} 
\]
for all $t \in (0, T)$.
\end{theorem}

\subsection{Error analysis for the semi-discrete problems}
\label{sub:3.3}

In the present section we develop an error analysis for the method for the semi-discrete problems. We introduce the energy projection $\mathcal{P}^{\nabla} \colon H^1_0(\Omega) \to W_h$ defined by
\begin{equation}
\label{proiezioneenergia}
\left \{
\begin{aligned}
& \text{find $\mathcal{P}^{\nabla}  u \in W_h$ such that} \\
& a_h(\mathcal{P}^{\nabla}  u, \, v_h) = a(u, \, v_h) \qquad \text{for all $v_h \in W_h$} 
\end{aligned}
\right .
\end{equation}
and the $L^2$-projection $\mathcal{P}^{0} \colon L^2(\Omega) \to W_h$ defined by
\begin{equation}
\label{proiezionel2}
\left \{
\begin{aligned}
& \text{find $\mathcal{P}^{0}  u \in W_h$ such that} \\
& m_h(\mathcal{P}^{0}  u, \, v_h) = (u, \, v_h)_{L^2(\Omega)} \qquad \text{for all $v_h \in W_h$.} 
\end{aligned}
\right .
\end{equation} 
The following approximation results hold (see \cite{vaccabeirao} for the proof)
\begin{lemma}
\label{lemma:stimeenergia}
Let $u \in H^1_0(\Omega) \cap H^{k+1}(\Omega)$.  Then there exists a unique function $\mathcal{P}^{\nabla} u \in W_h$ verifying 
\begin{equation}
\label{eq:energiah1}
|\mathcal{P}^{\nabla} u - u|_1 \leq C \, h^k \,  |u|_{k+1}. 
\end{equation}
Moreover, if the domain  $\Omega$ is convex, the following bound holds
\begin{equation}
\label{eq:energial2}
\|\mathcal{P}^{\nabla} u - u\|_0 \leq C \,  h^{k+1} \, |u|_{k+1}, 
\end{equation}
where   $C$ is an $h$-independent constant (depending only on $\alpha_*$ and $\alpha^*$).
\end{lemma}
For the $L^2$-projection we have the following lemma
\begin{lemma}
\label{lemma:stimal2}
Let $u \in H^{k+1}(\Omega)$.  Then there exists a unique function $\mathcal{P}^{0} u \in W_h$ verifying 
\begin{equation}
\label{eq:l2l2}
\|\mathcal{P}^{0} u - u\|_0 \leq C \, h^{k+1} \,  |u|_{k+1}, 
\end{equation}
where  $C$ is an $h$-independent constant (depending only on $\beta_*$ and $\beta^*$).
\end{lemma}
\begin{proof}
For the existence and uniqueness of $\mathcal{P}^0 u$  it is sufficient to observe that $\mathcal{P}^0 u$ is the solution of the variational problem \eqref{proiezionel2}. Since the bilinear form $m_h(\cdot, \, \cdot)$ is continuous and coercive and the functional  $(u, \, \cdot)_{L^2(\Omega)}$ is continuous on $W_h$, the previous problem has a unique solution.
Now, let $u_I$ the interpolant function of $u$ in the virtual space $W_h$ (see Proposition \ref{prp:mora}) and let $u_{\pi}$ the piecewise polynomial approximation of $u$ (see Theorem \ref{thm:scott}).   Let us set $\delta_h := \mathcal{P}^0 u - u_I$. Recalling the stability and consistency properties in Section \ref{sub:3.1}, some simple algebra yields 
\[
\begin{split}
\beta_*\|\delta_h\|^2_0 &= \beta_* (\delta_h, \, \delta_h)_{L^2(\Omega)} \leq m_h(\delta_h, \, \delta_h) = m_h(\mathcal{P}^0 u, \, \delta_h) - m_h(u_I, \,\delta_h) \\
&= m_h(\mathcal{P}^0 u, \, \delta_h) - \sum_{E \in \mathcal{T}_h} m_h^E(u_I, \, \delta_h) \\
&= (u, \, \delta_h)_{L^2(\Omega)} - \sum_{E \in \mathcal{T}_h} \left(m_h^E(u_I - u_{\pi}, \, \delta_h) + m_h^E(u_{\pi}, \, \delta_h)\right) \\
&= (u, \, \delta_h)_{L^2(\Omega)} - \sum_{E \in \mathcal{T}_h} \left(m_h^E(u_I - u_{\pi}, \, \delta_h) + (u_{\pi}, \, \delta_h)_{L^2(E)}\right) \\
&= (u, \, \delta_h)_{L^2(\Omega)} - \sum_{E \in \mathcal{T}_h} \left(m_h^E(u_I - u_{\pi}, \, \delta_h) + (u_{\pi} - u, \, \delta_h)_{L^2(E)}+ ( u, \, \delta_h)_{L^2(E)}\right) \\
&= (u, \, \delta_h)_{L^2(\Omega)} - \sum_{E \in \mathcal{T}_h} \left(m_h^E(u_I - u_{\pi}, \, \delta_h) + (u_{\pi} - u, \, \delta_h)_{L^2(E)} \right)- ( u, \, \delta_h)_{L^2(\Omega)} \\
&=  \sum_{E \in \mathcal{T}_h} \left(m_h^E( u_{\pi} - u_I, \, \delta_h) + (u - u_{\pi}, \, \delta_h)_{L^2(E)}\right).
\end{split}
\]
Therefore
\[
\beta_* \, \|\delta_h\|^2_0 \leq \beta^* \, \|u_{\pi} - u_I\|_{0} \, \|\delta_h\|_0 + \|u - u_{\pi} \|_{0} \, \|\delta_h\|_0,
\]
and thus 
\[
\begin{split}
\|\mathcal{P}^0 u - u\|_0 &\leq \|\mathcal{P}^0 u - u_I\|_0 +\|u_I - u\|_0 = \|\delta_h\|_0 + \|u_I - u\|_0 \leq C \, \left( \|u_I - u\|_0 + \|u_{\pi} - u\|_{0}\right).
\end{split}
\]
By  bounds \eqref{eq:interpolation} and \eqref{eq:scottpre} we can conclude that
\[
\|\mathcal{P}^0 u - u\|_0 \leq C \, h^{k+1} \, |u|_{k+1}.
\]
\end{proof}
The previous lemma allows us to derive the following error estimate. 

\begin{theorem}
\label{thm:first}
Let $u$ be the solution of  problem \eqref{eq:hyper var} and let us assume that $u \in C^2(0, T; \, H_0^1)$ and that $u_0$, $z_0$, $u_t$, $u_{tt}$ and $f(t)$ are in $H^{k+1}(\Omega)$. 
Let $u_h$ be the solution of  problem \eqref{eq:hyper vir}, then for all $t \in (0, T)$ it holds that
\begin{multline}
\label{eq:first}
|u_h(t) - u(t)|_1 + \|u_{h,t}(t) - u_t(t)\|_{0}  \leq C \left( |u_{h, 0} -  u_{ 0} |_1 + \|z_{h, 0} - z_0 \|_{0}  \right) +  \\
+ C \, h^k \left( |u_0|_{k+1} + |u_t(t)|_{k+1}   + h \, |z_0|_{k+1} + h\, |u_{tt}(t)|_{k+1} + h \, |f(t)|_{L^1(0,t,\, L^2(\Omega))}    \right)
\end{multline}
\end{theorem}
\begin{proof}
Let us set
\begin{equation}
\label{eq:thetarho}
u_h(t) - u(t)   = \left( u_h(t) - \mathcal{P}^{\nabla} u(t) \right) + \left(\mathcal{P}^{\nabla} u( t) -  u( t) \right)  =: \theta(t) + \rho( t),
\end{equation}
which are then estimated separately. The term $\rho(t)$ is the error generated by the energy projection. Using Lemma \ref{lemma:stimeenergia}, for all $t \in (0,T)$ we easily have
\begin{equation}
\label{eq:rhoh1}
\begin{split}
|\rho(t)|_1 &= |\mathcal{P}^{\nabla} u( t) - u( t)|_1 \leq C \, h^{k} \, |u(t)|_{k+1}\\
& \leq C \, h^{k} \, \left( |u_0|_{k+1} + \int_0^t |u_t(s)|_{k+1}\, {\rm d}x\right) =
C \, h^{k} \, \left( |u_0|_{k+1} +  |u_t|_{L^1(0,t,H^{k+1}(\Omega))}\right)
\end{split}
\end{equation}
in the same way
\begin{equation}
\label{eq:rhol2}
\|\rho_t(t)\|_0 = C \, h^{k+1} \, \left( |z_0|_{k+1} +  |u_{tt}|_{L^1(0,t,H^{k+1}(\Omega))}\right).
\end{equation}
In order to bound the term $\theta( t)$, we observe that, by \eqref{eq:hyper var}, definition \eqref{proiezioneenergia} and using that the derivative with respect to time commutes with the energy projection, for all $v_h \in W_h$ and for all $t \in (0,T)$ it holds 
\begin{equation}
\label{eq:theta}
\begin{split}
m_h(\theta_{tt}(t), \, v_h) &+ a_h(\theta( t), \, v_h) = \langle f_h(t), \, v_h \rangle - m_h\left( \frac{d^2}{dt^2} \mathcal{P}^{\nabla} u (t), \, v_h\right) - a_h(\mathcal{P}^{\nabla} u(t), \, v_h)  \\
& =  \langle f_h(t), \, v_h\rangle - m_h(\mathcal{P}^{\nabla} u_{tt}(t), \, v_h) - a(u(t),  \, v_h)   \\
& =  \langle f_h(t), \, v_h \rangle - \langle f(t), \, v_h\rangle + (u_{tt}( t), \, v_h)_{L^2(\Omega)} - m_h(\mathcal{P}^{\nabla} u_{tt}(t), \, v_h)     \\
& =  \langle f_h(t) - f(t), \, v_h \rangle + \left( (u_{tt}(t), \, v_h)_{L^2(\Omega)} - m_h(\mathcal{P}^{\nabla} u_{tt}(t), \, v_h) \right) \\
& =: \langle \phi(t), \, v_h \rangle + \langle \eta(t), \, v_h \rangle.
\end{split}
\end{equation}
Then the function $\theta$ solves the problem
\begin{equation}
\label{eq:thetahyp*}
\left \{
\begin{aligned}
& \text{find $\theta \in C^0(0, T; \, W_h)) \cap C^1(0,T; \, W_h)
$, such that} \\
& m_h( \theta_{tt}(t), \, v_h ) + a_h( \theta(t), \, v_h) =  \langle \phi(t) + \eta(t), \, v_h \rangle \quad \text{for all $v_h \in W_h$,  for a.e. $t$  in $(0,T)$} \\
& \theta(0) = u_{h, 0} - \mathcal{P}^{\nabla} u_0, \qquad \theta_t(0) = z_{h,0} - \mathcal{P}^{\nabla} z_0
\end{aligned}
\right.
\end{equation}
and, for Theorem \ref{esistenza}, it holds that
\begin{multline}
\label{eq:estimatetheta}
\left( a_h(\theta(t), \, \theta(t)) + \|\theta_t(t)\|_{h,0}^2 \right)^{\frac{1}{2}} \leq \left( a_h(\theta(0), \, \theta(0)) + \|\theta_t(0)\|^2_{h, 0} \right)^{\frac{1}{2}} +\\
+ |\phi |_{L^1(0,t,\, L^2(\Omega))} + |\eta|_{L^1(0,t,\, L^2(\Omega))}. 
\end{multline}
We can observe that the term $\phi$ can be bounded as follows
\begin{equation}
\label{eq:I+}
\begin{split}
\langle \phi(t), \, v_h \rangle & = \langle f_h( t) - f(t), \, v_h \rangle = \sum_{E \in \mathcal{T}_h} \int_E (\Pi^{0, E}_k f(t) - f( t) ) \, v_h \, {\rm d}x \\
& \leq  \sum_{E \in \mathcal{T}_h} C \, h^{k+1} \, |f(t)|_{k+1, E} \, \|v_h\|_{0, E} = C \, h^{k+1} \, |f(t)|_{k+1} \, \|v_h\|_{0}.
\end{split}
\end{equation}
For the term $\eta$, using the consistency and stability properties of the bilinear form $m_h(\cdot, \cdot)$, we can obtain
\begin{equation}
\label{eq:II+}
\begin{split}
\langle \eta(t), \, v_h \rangle &= (u_{tt}( t), \, v_h)_{L^2(\Omega)} - m_h(\mathcal{P}^{\nabla} u_{tt}(t), \, v_h)   = \\
&= \sum_{E \in \mathcal{T}_h} \biggl(  (u_{tt}(t), \, v_h)_{L^2(E)} - m_h^E(\mathcal{P}^{\nabla} u_{tt}(t), \, v_h)  \biggr) \\
&= \sum_{E \in \mathcal{T}_h} \biggl(  (u_{tt}(t) -  \Pi^{0, E}_{k} u_{tt}(t), \, v_h)_{L^2(E)} - m_h^E(\mathcal{P}^{\nabla} u_{tt}(t) - \Pi^{0, E}_{k} u_{tt}(t), \, v_h)   \biggr)\\
&= \sum_{E \in \mathcal{T}_h} \biggl(  (u_{tt}(t) -  \Pi^{0, E}_{k} u_{tt}(t), \, v_h)_{L^2(E)} + m_h^E(\Pi^{0, E}_{k} u_{tt}(t) -\mathcal{P}^{\nabla} u_{tt}(t), \, v_h)   \biggr)\\
& \leq  \sum_{E \in \mathcal{T}_h} C \, \left( \|u_{tt}( t) - \Pi^{0, E}_{k} u_{tt}(t)\|_{0, E} + \|\Pi^{0, E}_{k} u_{tt}( t) - \mathcal{P}^{\nabla} u_{tt}(t)\|_{0, E} \right) \|v_h\|_{0,E}\\
& \leq C \, h^{k+1} \, |u_{tt}( t)|_{k+1} \, \|v_h\|_0.
\end{split}
\end{equation}
For the initial data we simply have
\begin{equation}
\label{eq:indatah1}
\begin{split}
a_h(\theta(0), \, \theta(0)) & \leq \alpha^* \, |\theta(0)|^2_1 = \alpha ^* |u_{h, 0} - \mathcal{P}^{\nabla} u_0|^2_1 \leq C \, \left( |u_{h,0} - u_0|^2_1 +   |u_{0} - \mathcal{P}^{\nabla} u_0|^2_1\right) \\
& \leq C \, \left( |u_{h,0} - u_0|^2_1 + h^{2k} \, |u_0|^2_{k+1} \right) 
\end{split}
\end{equation}
and similarly
\begin{equation}
\label{eq:indatal2}
 \|\theta_t(0)\|_{h, 0}^2 \leq \beta^* \, \|z_{h,0} - \mathcal{P}^{\nabla} z_0\|_0^2  \leq C \, \left( \|z_{h,0} - z_0\|_0^2 + h^{2(k+1)} \, |z_0|^2_{k+1}\right).
\end{equation}
Then, by collecting \eqref{eq:I+}, \eqref{eq:II+}, \eqref{eq:indatah1}, \eqref{eq:indatal2}, in \eqref{eq:estimatetheta}
\begin{multline}
\label{eq:estimatethetabis}
\left( a_h(\theta(t), \, \theta(t)) + \|\theta_t(t)\|_{h,0}^2 \right)^{\frac{1}{2}} \leq C \left( |u_{h,0} - u_0|_1 +  \|z_{h,0} - z_0\|_0  \right)+ \\  
+ C \, \left( h^{k} \, |u_0|_{k+1}  +  h^{k+1} \,|z_0|_{k+1} + h^{k+1} \, |u_{tt}( t)|_{L^1(0, t, H^{k+1}(\Omega))} + h^{k+1} \, |f(t)|_{L^1(0, t, H^{k+1}(\Omega))}  \right).
\end{multline} 
Finally, from \eqref{eq:rhoh1}, \eqref{eq:rhol2} and \eqref{eq:estimatethetabis} we get the thesis.
\end{proof}

\begin{remark}
We observe that from the estimate \eqref{eq:first}, we immediately obtain the $H^1$ semi-norm estimate of the error between the semi-discrete solution and the continuous solution, i.e.
\begin{equation}
\label{eq:estimateh1}
|u_h(t) - u(t)|_1  \leq C \left( |u_{h, 0} -  u_{ 0} |_1  + h^k \, |u_0|_{k+1} + h^k\, |u_t(t)|_{k+1}  + O(h^{k+1})\right). 
\end{equation}
\end{remark}
For the $L^2$ estimate of the error we can state the following theorem.
\begin{theorem}
\label{thm:semidiscretol2}
Under the assumptions of the Theorem \ref{thm:first}, for all $t \in (0,T)$ it holds that
\begin{multline}
\label{eq:stimadiscretol2}
\|u_{h,t}(t) - u_t(t)\|_{0}  \leq C \left( |u_{h, 0} -  u_{ 0} |_1 + \|v_{h, 0} - z_0 \|_{0}  \right) +  \\
+ C \, h^{k+1} \left( |u_0|_{k+1} +  |z_0|_{k+1} +  |u_{tt}|_{L^2(0,t, \, H^{k+1}(\Omega))} +  |f|_{L^2(0,t, \, H^{k+1}(\Omega))}  \right)
\end{multline}
\end{theorem}
\begin{proof}
As before in \eqref{eq:thetarho}, let us set $u_h(t) - u(t) = \theta(t) + \rho( t)$. The $L^2$ norm of the term $\rho(t)$ can be bounded as in \eqref{eq:rhol2}. Now since $\theta(t)$ solves the PDE \eqref{eq:thetahyp*}, recalling \eqref{eq:hyper vir}, we have that
\[
\theta(t) = \sum_{n=1}^{Ndof} \gamma_n \, w_h^{(n)}
\]
where
\begin{multline}
\label{eq:gamma_n}
\gamma_n = m_h(\theta(0), \, w_h^{(n)}) \, \cos (\mu_h^{(n)} t)  + \frac{1}{\mu_h^{(n)} } m_h(\theta_t(0), \,w_h^{(n)}) \, \sin (\mu_h^{(n)} t)  +\\  
  + \frac{1}{\mu_h^{(n)} } \, \int_0^t \langle \phi(s) + \eta(s), \, w_h^{(n)} \rangle \,  \sin (\mu_h^{(n)} (t-s))  \, {\rm d}s.
\end{multline}
Considering that $\set{w_h^{(n)}}_{1, \dots, Ndof}$ is an orthonormal basis of $W_h$ with respect to $m_h(\cdot, \cdot)$ it holds that
\begin{equation}
\label{eq:sumgamma}
\|\theta(t)\|^2_{0,h} = m_h(\theta(t), \, \theta(t)) = \sum_{n=1}^{Ndof} |\gamma_n|^2.
\end{equation}
Some simple computations yield
\begin{gather*}
0 \leq \mu_h^{(n)} \leq \epsilon \qquad \text{then} \qquad  \frac{ \sin (\mu_h^{(n)} t) }{\mu_h^{(n)}} \leq C \, t \quad \text{for all $t \in (0,T)$,} \\
\mu_h^{(n)} \geq \epsilon  \qquad \text{then} \qquad  \frac{ \sin (\mu_h^{(n)} t) }{\mu_h^{(n)}} \leq C  \quad \text{for all $t \in (0,T)$,} \\
\end{gather*}
for $\epsilon$ small enough. Therefore, from Jensen inequality, we get 
\[
|\gamma_n|^2 \leq C(t) \, \left(m_h(\theta(0), \, w_h^{(n)})^2  +  \, m_h(\theta_t(0), \, w_h^{(n)})^2 +  \int_0^t \left( \langle \phi(s) + \eta(s), \, w_h^{(n)} \rangle \right)^2  {\rm d}s  \right),
\]
where $C(t) := \max \{1, t^2\}$. From \eqref{eq:sumgamma} it follows that
\begin{equation}
\label{eq:thetal2l2}
\begin{split}
\|\theta(t)\|^2_{0,h} & \leq C(t) \,  \sum_{n=1}^{Ndof} \left(m_h(\theta(0), \, w_h^{(n)})^2  +   m_h(\theta_t(0), \, w_h^{(n)})^2 +  \int_0^t \left(  \langle \phi(s) + \eta(s), \, w_h^{(n)} \rangle  \right)^2  {\rm d}s \right) \\
& \leq C(t) \left( \|\theta(0)\|^2_{0,h}  + \|\theta_t(0)\|^2_{0,h} +   \int_0^t \sum_{n=1}^{Ndof} \left(  \langle \phi(s) + \eta(s), \, w_h^{(n)} \rangle \right)^2 {\rm d}s  \right).
\end{split}
\end{equation}
Now, from the definition \eqref{proiezionel2} we can set
\[
\langle \phi(s) + \eta(s), \, w_h^{(n)} \rangle = m_h(\mathcal{P}^{0} ( \phi(s) + \eta(s)), \,  w_h^{(n)}) 
\]
therefore, since $\set{w_h^{(n)}}_{1, \dots, Ndof}$ is an orthonormal basis of $W_h$ with respect to $m_h(\cdot, \cdot)$, we get
\[
\begin{split}
\sum_{n=1}^{Ndof} \left(  \langle \phi(s) + \eta(s), \, w_h^{(n)} \rangle \right)^2 &= \sum_{n=1}^{Ndof} \left( m_h(\mathcal{P}^{0} ( \phi(s) + \eta(s)), \,  w_h^{(n)}) \right)^2 = \|\mathcal{P}^{0} ( \phi(s) + \eta(s))\|^2_{0, h}.
\end{split}
\]
It is easy to see that, from \eqref{proiezionel2} and from the equivalence between the discrete and the continuous $L^2$ norm, we obtain
\[
\sum_{n=1}^{Ndof} \left(  \langle \phi(s) + \eta(s), \, w_h^{(n)} \rangle \right)^2 = \|\mathcal{P}^{0} ( \phi(s) + \eta(s))\|^2_{0, h} \leq C \, \| \phi(s) + \eta(s)\|^2_{0},
\]
therefore from \eqref{eq:thetal2l2}, according with estimates \eqref{eq:I+} and \eqref{eq:II+} we take
\begin{equation}
\label{eq:thetadef}
\|\theta(t)\|^2_{0,h} \leq C(t) \left( \|\theta(0)\|_{0,h}^2 + \|\theta_t(0)\|_{0,h}^2 + h^{k+1} \, |u_{tt}|_{L^2(0,t, \, H^{k+1}(\Omega))} + h^{k+1} \, |f|_{L^2(0,t, \,H^{k+1}(\Omega))}  \right).
\end{equation}
Collecting \eqref{eq:rhol2} and \eqref{eq:thetadef}, similar argument of Theorem \ref{thm:first} give the thesis.
\end{proof}

\section{Fully discrete problems}
\label{sec:4}

Since the error analysis of the time discretisation follows a standard procedure, we focus mainly on the error between the continuous problem \eqref{eq:hyper var} and the semi-discrete problem \eqref{eq:hyper vir}. In this section we show an example of analysis for the fully discrete case.

Theoretically, the error generated by a fully discrete scheme has two components: the error due to the spatial discretization  depending on the mesh size $h$, and the  error created by the time integrator depending on the time step size $\tau$. In particular let
$\set{u_h^n}_{n=0, \dots, N}$ be the sequence generated by a time integrator method $\mathcal{I}$ for the ODE \eqref{eq:hyper vir}, with $u_h^n \approx u_h(\cdot, t_n)$, $t_n = n \tau$, for $n=0,\dots, N$ and $\tau = T/N$. Then we expect that
\begin{equation}
\label{eq:teorica}
\|u_h^n - u(\cdot, t_n)\|_0 \leq C_1 \, h^{k+1} +  C_2 \, \tau^p,
\end{equation}
where $p$ is the order of the method $\mathcal{I}$, and $C_1$ and $C_2$ are two $h$ and $\tau$ independent constants. 

As already mentioned, since the novelty of the present paper is the spatial discretisation, we focus mainly on the first (spatial) source of error, as shown in Theorems \ref{thm:first} and \ref{thm:semidiscretol2}. Nevertheless, in order to detail the behaviour of the method, we here consider the case of the \textbf{Newmark method} and the \textbf{Bathe method} coupled with the VEM discretisation \eqref{eq:hyper vir}. 
The Newmark method  (see \cite{newmark, raviart})  for the ODE \eqref{eq:hyper vir}, is defined by
\begin{equation}
\label{eq:hyper dis}
\left\{
\begin{aligned}
&m_h\left( \frac{u_h^{n+1} - u_h^n - \tau \, z_h^n}{\tau^2}, v_h\right) + a_h \left(\beta u_h^{n+1} + \left( \frac{1}{2} - \beta \right) u_h^n, \, v_h \right) =  \langle \beta f_h^{n+1} + \left( \frac{1}{2} - \beta \right)f_h^n,   v_h \rangle   \\
&m_h\left( \frac{z_h^{n+1} - z_h^n}{\tau}, v_h\right) + a_h \left(\gamma u_h^{n+1} + ( 1 - \gamma) u_h^n, \, v_h \right) = \langle \gamma f_h^{n+1} + (1- \gamma)f_h^n, \, v_h \rangle  \\
&u_h^0 = u_{h,0}, \qquad z_h^0 =z_{h,0}
\end{aligned}
\right.
\end{equation}
or equivalently
\begin{equation}
\label{eq:hyper disbis}
\left\{
\begin{aligned}
&m_h\left( \frac{u_h^{n+2} - 2u_h^{n+1} + u_h^n}{\tau^2}, v_h\right) + a_h \left(\beta u_h^{n+2} + \left( \frac{1}{2} - 2\beta + \gamma\right) u_h^{n+1} + \left( \frac{1}{2} + \beta - \gamma\right) u_h^{n}, v_h \right) \\
& \qquad \qquad \qquad \qquad \qquad \qquad \qquad = \langle \beta f_h^{n+2} + \left( \frac{1}{2} - 2\beta + \gamma\right) f_h^{n+1} + \left( \frac{1}{2} + \beta - \gamma\right) f_h^{n}, v_h \rangle   \\
&m_h\left( \frac{u_h^{1} - u_{h,0} - \tau \, z_{h,0}}{\tau^2}, v_h\right) + a_h \left(\beta u_h^{1} + \left( \frac{1}{2} - \beta \right) u_{h,0}, \, v_h \right) =  \langle \beta f_h^{1} + \left( \frac{1}{2} - \beta \right)f_h^0, v_h  \rangle \\
& u_h^0 = u_{h, 0}
\end{aligned}
\right.
\end{equation}
where $f_h^n = f_h(t_n)$ for $n=0,\dots, N$, while $\beta \geq 0$ and $\gamma \geq 1/2$ are free parameters that still can be chosen. From the literature, we recall the following facts:
\begin{itemize}
\item \textbf{convergence}: The Newmark scheme is at least of order one; the order two is achieved only for the choice $\gamma = 1/2$;
\item \textbf{stability}: The second-order Newmark scheme with $\gamma = 1/2$ is unconditionally stable for $\beta \geq 1/4$, whereas for $0 \leq \beta  < 1/4$ the time step $\tau$ has to be restricted by the \textbf{CFL condition}
\[
\lambda_h^{(Ndof)} \, \tau^2 \leq \frac{4}{1 - 4 \beta} (1 - \epsilon), \qquad \text{for $\epsilon \in (0, 1)$.}
\]
\end{itemize}

It is well known (see \cite{assessment}) that the widely used Newmark trapezoidal rule (corresponding to $\gamma = 1/2$ and $\beta = 1/4$) does not present the numerical dumping, i.e. this technique is affected by spurious oscillations, especially for high wave numbers, that can severely ruin the accuracy of the solution.  The Bathe method \cite{bathe1, bathe2, bathe3, bathe4}  is, indeed,  quite effective in the solution of wave propagation problems. The Bathe method for ODE \eqref{eq:hyper vir} has the following linear multistep form  
\begin{equation}
\label{eq:hyper bathe}
\left\{
\begin{aligned}
& m_h\left( \frac{72 u_h^{n+1} - 144 u_h^{n+1/2} + 72 u_h^n}{\tau^2}, v_h \right) + a_h (8 u_h^{n+1} + 5 u_h^{n+1/2} + 5 u_h^{n}, v_h ) =\\
& \qquad \qquad \qquad \qquad \qquad \qquad \qquad \qquad \qquad \qquad \qquad \qquad \qquad \quad  \langle 8 f_h^{n+1} + 5 f_h^{n+1/2} + 5 f_h^{n}, v_h  \rangle \\
& m_h\left( \frac{16 u_h^{n+1/2} -  16 u_{h, 0} - \tau 8 z_{h, 0} }{\tau^2}, v_h \right) + a_h (u_h^{n+1/2} -  u_{h, 0}, v_h ) = \langle f_h^{1/2} - f_h^{0}, v_h  \rangle \\
& u_h^0 = u_{h, 0}.
\end{aligned}
\right.
\end{equation}
For the Bathe method we have the following properties:
\begin{itemize}
\item \textbf{convergence}: The Bathe method has order two; 
\item \textbf{stability}: The method has no parameter to choose or adjust, by
the analyst, for specific analysis cases. The scheme is stable even in large deformation and long time response solutions when the trapezoidal rule fails.
\end{itemize}


\section{Numerical Tests}
\label{sec:5}
In this section we present two numerical experiments to test the practical performance of the method. 
In the first  test we compute the error in the $H^1$ semi-norm and in  $L^2$ norm for a given hyperbolic problem. We investigate also the behaviour of the method when we use  a non stabilized form $m_h$. 
The second experiment investigates the performance of the Bathe method in the solution of wave propagation problems compared with the Newmark trapezoidal rule.
%

%


\begin{test}
\label{test1}
Let us consider the parabolic equation \eqref{eq:hyper var} where  the load term $f$, the initial data $u_0$ and $z_0$ are chosen in accordance with the exact solution 
\begin{equation}
\label{eq:soluzione1}
u(t,x_1, x_2) =\sin(t^2) \, \sin(\pi x_1)\sin(\pi x_2).
\end{equation}
In this test we consider the time interval $[0, 1]$ and  the domain $\Omega = [0,1] \times [0,1]$. We use the Voronoi meshes $\mathcal{V}_h$ (where $h= \frac{1}{5}2^{-i}$, with $i=0, \dots, 3$, is the mean value of the mesh size).  For the generation of the Voronoi meshes we used the code Polymesher \cite{TPPM12}. The adopted meshes are shown in Figure \ref{Figure1}.  

\begin{figure}[!h]
\centering
\includegraphics[scale=0.25, draft=false]{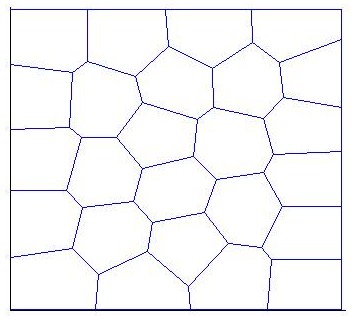}\quad
\includegraphics[scale=0.25, draft=false]{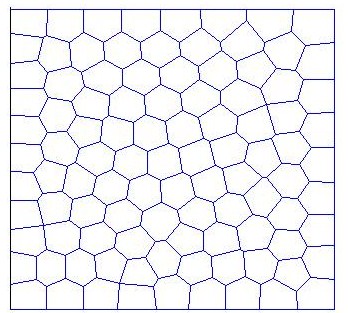}\quad
\includegraphics[scale=0.255, draft=false]{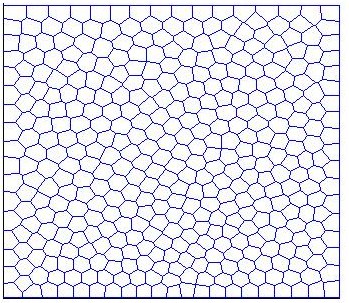}\quad
\includegraphics[scale=0.255, draft=false]{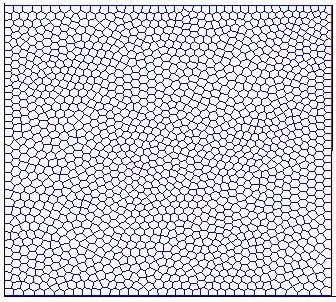}
\caption{Sequence of the adopted Voronoi meshes.}
\label{Figure1}
\end{figure}
The convergence of VEM technique  is evaluated in the discrete relative $H^1(\Omega)$ semi-norm $L^2(\Omega)$ norm of  $\delta_h := u^I - u_h$ where $u^I$ is the interpolant of the exact solution $u$, evaluate at the final time $T$, i.e. 
\[
E^1_{h,\tau} := \frac{|\delta_h|_{1, h}}{|u|_{1, h}}, \qquad E^1_{h,\tau} := \frac{\|\delta_h\|_{0, h}}{\|u\|_{0, h}}.
\]
We implement the fully discrete problem with the Newmark trapezoidal method coupled with the VEM discretisation for  the sequences of polygonal meshes $\mathcal{V}_h$. The orders of approximation are $k=1$ and $k=2$. In Tables \ref{table1} and \ref{table2} we show the values of the relative errors $E^1_{h,\tau}$ and $E^0_{h,\tau}$.

\begin{table}[!h]
\centering
\begin{tabular}{ll*{4}{c}}
\toprule
&              & $\tau = 1/5$        & $\tau =  1/10$      & $\tau = 1/20$       & $\tau = 1/40$       \\
\midrule
\multirow{4}*{$k=1$}
& $h=1/5$   & $3.528027e-02$ &  $3.402436e-02$ &  $3.394938e-02$ & $3.393157e-02$\\
& $h=1/10$ & $2.076253e-02$ &  $1.597628e-02$ &  $1.587273e-02$ & $1.590796e-02$ \\
& $h=1/20$ & $1.650362e-02$ &  $7.653534e-03$ &  $6.847137e-03$ & $6.841819e-03$ \\
& $h=1/40$ & $1.587896e-02$ &  $5.192064e-03$ &  $3.554348e-03$ & $3.452997e-03$ \\
\midrule
\multirow{4}*{$k=2$}
& $h=1/5$    & $7.563951e-02$ &  $1.034244e-01$ &  $6.890232e-02$ & $7.358191e-02$\\
& $h=1/10$  & $1.331417e-02$ &  $2.354252e-02$ &  $1.510415e-02$ & $1.727028e-02$ \\
& $h=1/20$  & $4.521837e-03$ &  $3.786975e-03$ &  $3.367110e-03$ & $4.267637e-03$ \\
& $h=1/40$  & $1.252141e-03$ &  $1.240103e-03$ &  $1.205809e-03$ & $9.117119e-04$ \\
\bottomrule
\end{tabular}
\caption{$E^1_{h, \tau}$ error for the meshes $\mathcal{V}_h$ for $k=1$ and $k=2$.}
\label{table1}
\end{table}

\begin{table}[!h]
\centering
\begin{tabular}{ll*{4}{c}}
\toprule
&              & $\tau = 1/5$        & $\tau =  1/10$      & $\tau = 1/20$       & $\tau = 1/40$       \\
\midrule
\multirow{4}*{$k=1$}
& $h=1/5$    & $1.525822e-02$ &  $1.097503e-02$ &  $1.123438e-02$ & $1.139287e-02$\\
& $h=1/10$  & $1.368594e-02$ &  $3.405071e-03$ &  $2.917899e-03$ & $3.241443e-03$ \\
& $h=1/20$  & $1.503495e-02$ &  $3.497376e-03$ &  $7.462431e-04$ & $7.032843e-04$ \\
& $h=1/40$  & $1.550013e-02$ &  $3.881783e-03$ &  $8.608672e-04$ & $1.784726e-04$ \\
\midrule
\multirow{4}*{$k=2$}
& $h=1/5$   & $1.737477e-02$  & $2.310717e-02$ &  $1.513967e-02$ & $1.647027e-02$\\
& $h=1/10$ & $1.432602e-03$  & $2.597125e-03$ &  $1.626653e-03$ & $1.893847e-03$ \\
& $h=1/20$ & $2.941979e-04$  & $2.088662e-04$ &  $1.864353e-04$ & $2.412417e-04$ \\
& $h=1/40$ & $1.679805e-04$ &  $4.656927e-05$ &  $3.446027e-05$ & $2.578508e-05$ \\
\bottomrule
\end{tabular}
\caption{$E^0_{h, \tau}$ error for the meshes $\mathcal{V}_h$ for $k=1$ and $k=2$.}
\label{table2}
\end{table}
In this test we notice that, recalling \eqref{eq:teorica}, the error due to the spatial discretization and the error generated by the time discretisation has the same weight. In particular, for small values of $\tau$, we can observe that we obtain the expected order of convergence in $h$. For big values of $h$, we have that the error is almost constant in $\tau$.

We consider the same hyperbolic problem \eqref{eq:soluzione1} and we study the behaviour of the VEM approximation using a non stabilized bilinear form $m_h(\cdot, \cdot)$ obtained using $\mathcal{R}^E \equiv 0$ in \eqref{eq:m_hlocalform}.
We consider as before the Newmark method coupled with the VEM approximation of order $k=1, 2$ for the usual sequences of Voronoi meshes $\mathcal{V}_h$. The results are illustrated in  Tables \ref{table3} and \ref{table4}, the errors being evaluated as usual in the $E^1_{h, \tau}$ and norm $E^0_{h, \tau}$.

\begin{table}[!h]
\centering
\begin{tabular}{ll*{4}{c}}
\toprule
&              & $\tau = 1/5$        & $\tau =  1/10$      & $\tau = 1/20$       & $\tau = 1/40$       \\
\midrule
\multirow{4}*{$k=1$}
& $h=1/5$    & $3.528854e-02$ &  $3.361425e-02$ &  $3.374885e-02$ & $3.380632e-02$\\
& $h=1/10$  & $2.071838e-02$ &  $1.592896e-03$ &  $1.582944e-02$ & $1.589366e-02$ \\
& $h=1/20$  & $1.650164e-02$ &  $7.648859e-03$ &  $6.843560e-03$ & $6.838834e-03$ \\
& $h=1/40$  & $1.550013e-02$ &  $5.192021e-03$ &  $3.553935e-03$ & $3.452686e-03$ \\
\midrule
\multirow{4}*{$k=2$}
& $h=1/5$    & $9.149928e-02$ &  $7.437253e-02$ &  $8.727999e-02$ & $7.917469e-02$\\
& $h=1/10$  & $2.372960e-03$ &  $1.841401e-02$ &  $1.649658e-02$ & $2.000535e-02$ \\
& $h=1/20$  & $5.509267e-03$ &  $3.211032e-03$ &  $3.362645e-03$ & $3.816907e-03$ \\
& $h=1/40$  & $1.318321e-03$ &  $6.451261e-04$ &  $8.166188e-04$ & $7.931607e-04$ \\
\bottomrule
\end{tabular}
\caption{$E^1_{h, \tau}$ error for the meshes $\mathcal{V}_h$ for $k=1$ and $k=2$.}
\label{table3}
\end{table}

\begin{table}[!h]
\centering
\begin{tabular}{ll*{4}{c}}
\toprule
&              & $\tau = 1/5$        & $\tau =  1/10$      & $\tau = 1/20$       & $\tau = 1/40$       \\
\midrule
\multirow{4}*{$k=1$}
& $h=1/5$    & $1.180067e-02$ &  $5.070831e-03$ &  $5.667866e-03$ & $6.001536e-03$\\
& $h=1/10$  & $1.351103e-02$ &  $2.650216e-03$ &  $1.988189e-03$ & $2.442523e-03$ \\
& $h=1/20$  & $1.502775e-02$ &  $3.467810e-03$ &  $5.932316e-04$ & $5.383329e-04$ \\
& $h=1/40$  & $1.549971e-02$ &  $3.880175e-03$ &  $8.536129e-04$ & $1.393609e-04$ \\
\midrule
\multirow{4}*{$k=2$}
& $h=1/5$   & $7.946409e-03$  & $8.532940e-03$ &  $8.380126e-03$ & $8.142198e-03$\\
& $h=1/10$ & $5.954846e-04$  & $6.483602e-04$ &  $5.820595e-04$ & $6.189287e-04$ \\
& $h=1/20$ & $1.649159e-04$  & $5.905337e-05$ &  $5.288382e-05$ & $5.507499e-05$ \\
& $h=1/40$ & $1.644686e-04$ &  $3.106303e-05$ &  $8.861383e-06$ & $5.766137e-06$ \\
\bottomrule
\end{tabular}
\caption{$E^0_{h, \tau}$ error for the meshes $\mathcal{V}_h$ for $k=1$ and $k=2$.}
\label{table4}
\end{table}
Comparing the results of Tables  \ref{table1} and \ref{table2} with those of Tables \ref{table3} and \ref{table4}, we can observe that the errors generated by the VEM method with original and reduced bilinear form are indeed very close, thus showing the good behaviour of the proposed alternative. However we remark that, as observed in \cite{vaccabeirao}, if we take a non stabilized bilinear form, the maximum discrete eigenvalue $\lambda_h^{(Ndof}$ is not bounded since the mass matrix is possibly singular. Then we have to be careful if we use a Newmark method with $\beta < 1/4$, in particular if we want to suppress the high frequency spurious waves.
\end{test}

\begin{test}
\label{test2}
Let us consider the parabolic equation \eqref{eq:hyper var} where the initial displacement and the initial velocity are zero, and 
the load term $f$ is
\begin{equation}
\label{eq:loadterm}
f(t,x_1, x_2) := 
\left \{
\begin{aligned}
& 100  \qquad & \text{for $t < 0.1$ and $(x_1, x_2) = (0.05, 0.05)$,} \\
& 0                              \qquad & \text{otherwise,}
\end{aligned}
\right .
\end{equation}
We consider the final time $T = 1.2$ and the domain $\Omega = [0, 1] \times [0, 1]$. Let $\mathbf{CFL} = \frac{\tau}{h}$, then we test the VEM technique with $k=1$ coupled with the Bathe method and the Newmark trapezoidal rule with different values of $CLF$, in particular we use a square decomposition of the domain with $h=1/100$ and we consider $\tau = 1/20, 1/40, 1/80$.
In Figure \ref{Figure2}, \ref{Figure3} \ref{Figure4}  we compare the discrete displacement $u_h$, and the discrete velocity $u_{h, t}$ along the diagonal $(0, 0)$ -$(1, 1)$ calculated using the Bathe method and the trapezoidal rule at final time. We observe that for small values of CLF the Newmark trapezoidal rule gives spurious oscillations.

\begin{figure}[!h]
\center{
\includegraphics[scale=0.20, draft=false]{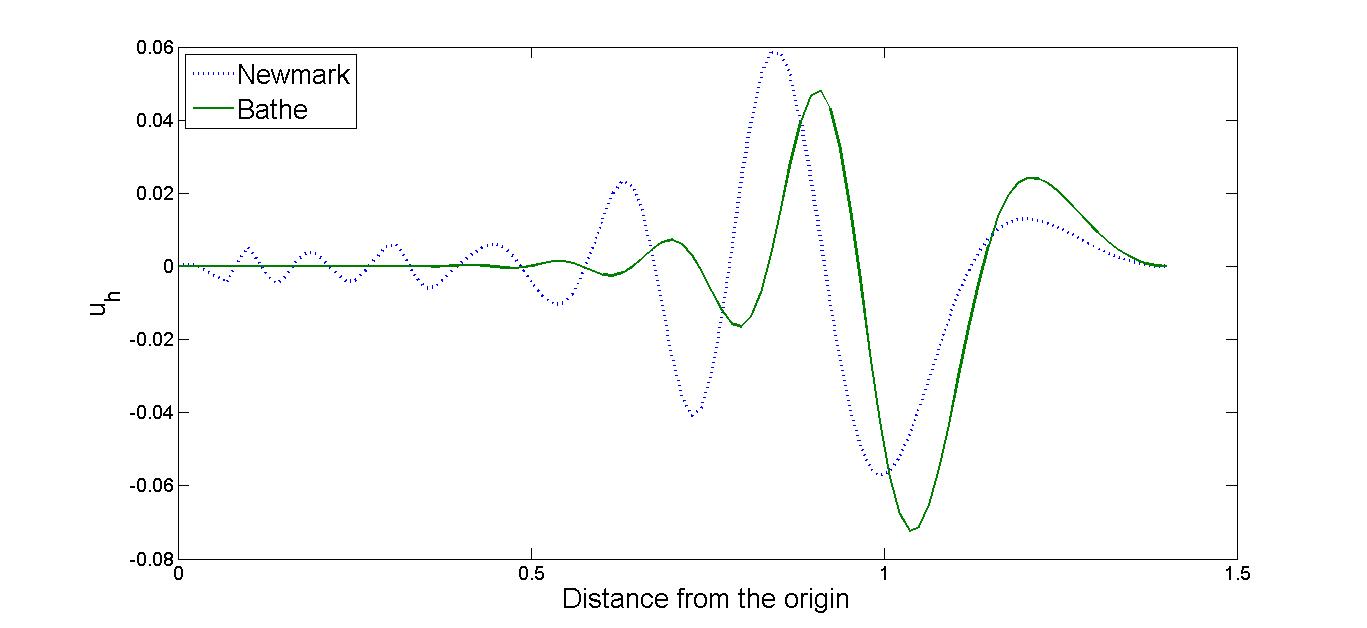} \quad
\includegraphics[scale=0.20, draft=false]{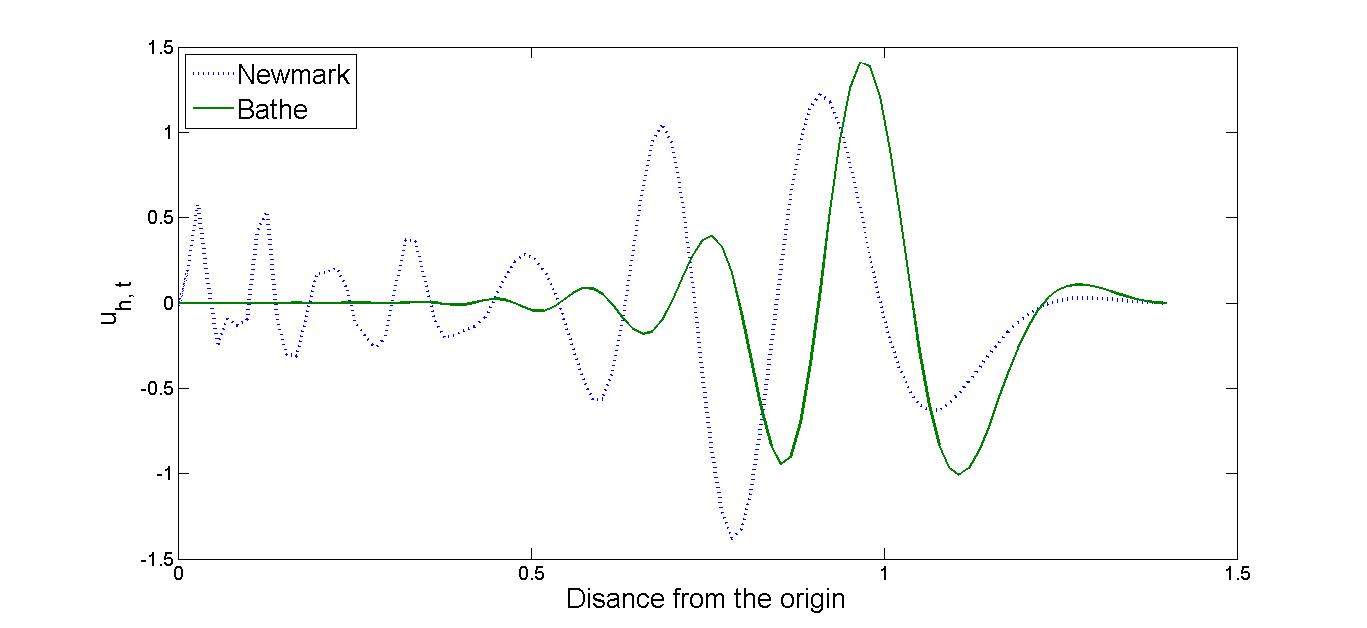} 
\caption{Displacement variations and velocity variations along the diagonal with $\tau = 1/20$.}
\label{Figure2}
}
\end{figure}
\begin{figure}[!h]
\center{
\includegraphics[scale=0.20, draft=false]{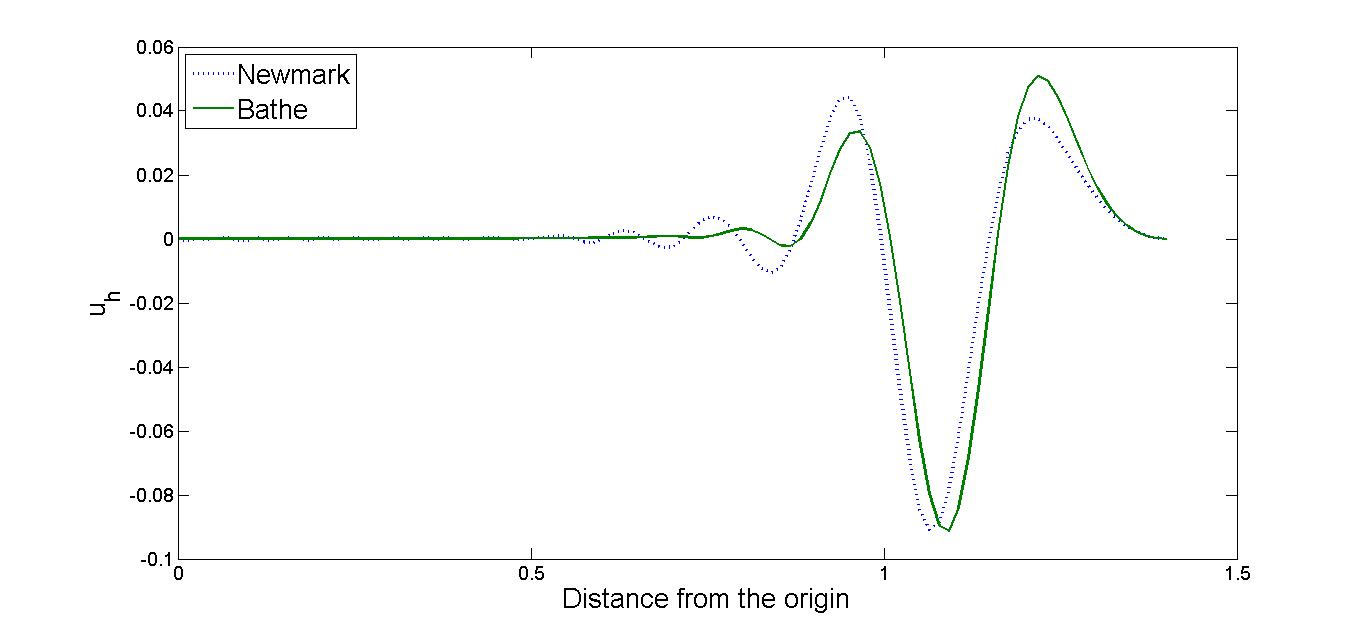} \quad
\includegraphics[scale=0.20, draft=false]{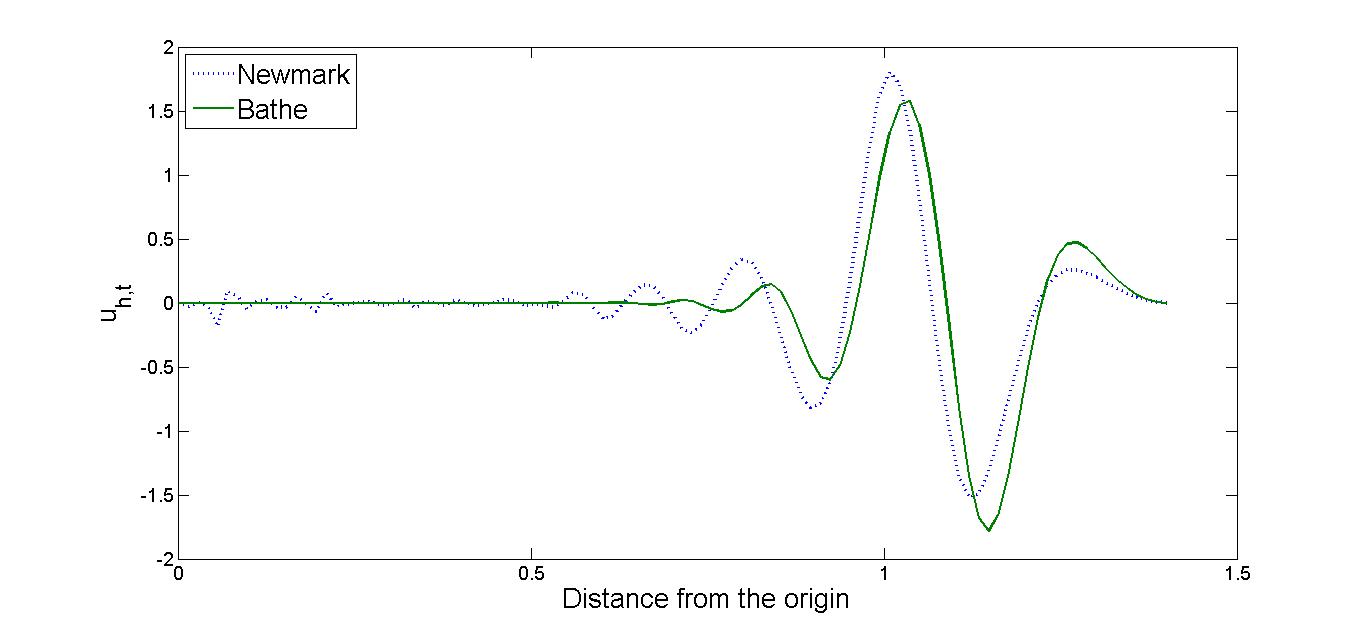} 
\caption{Displacement variations and velocity variations along the diagonal with $\tau = 1/40$.}
\label{Figure3}
}
\end{figure}
\begin{figure}[!h]
\center{
\includegraphics[scale=0.20, draft=false]{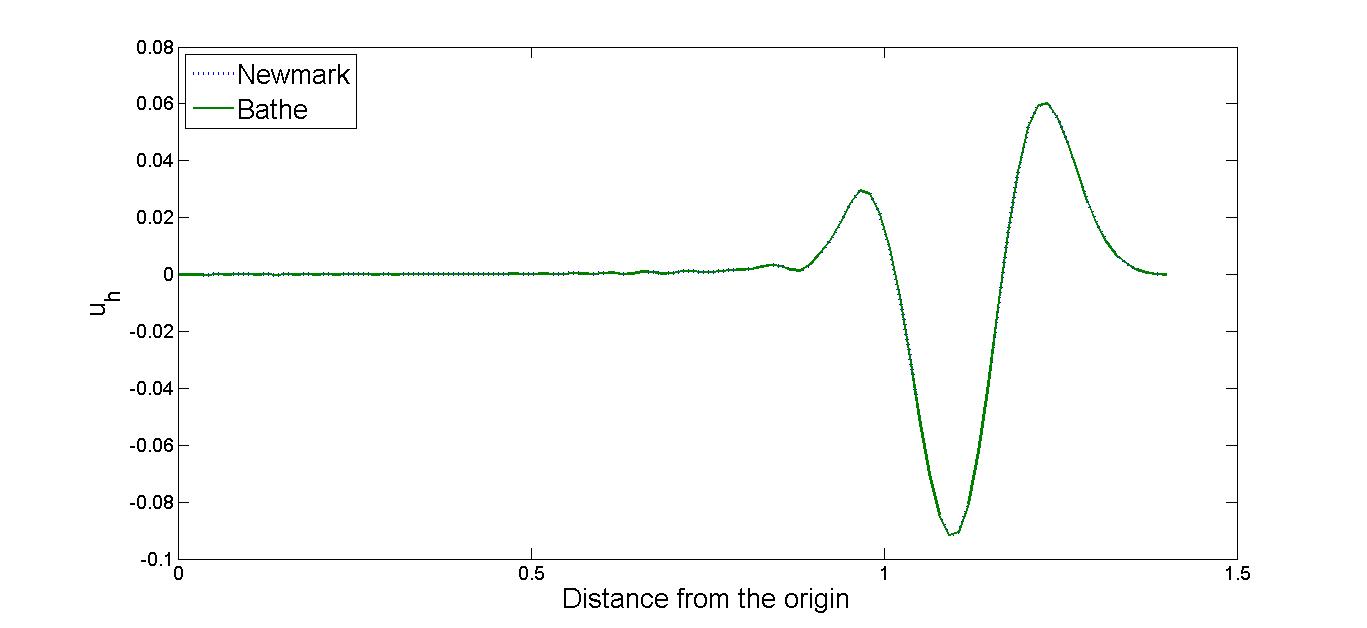} \quad
\includegraphics[scale=0.20, draft=false]{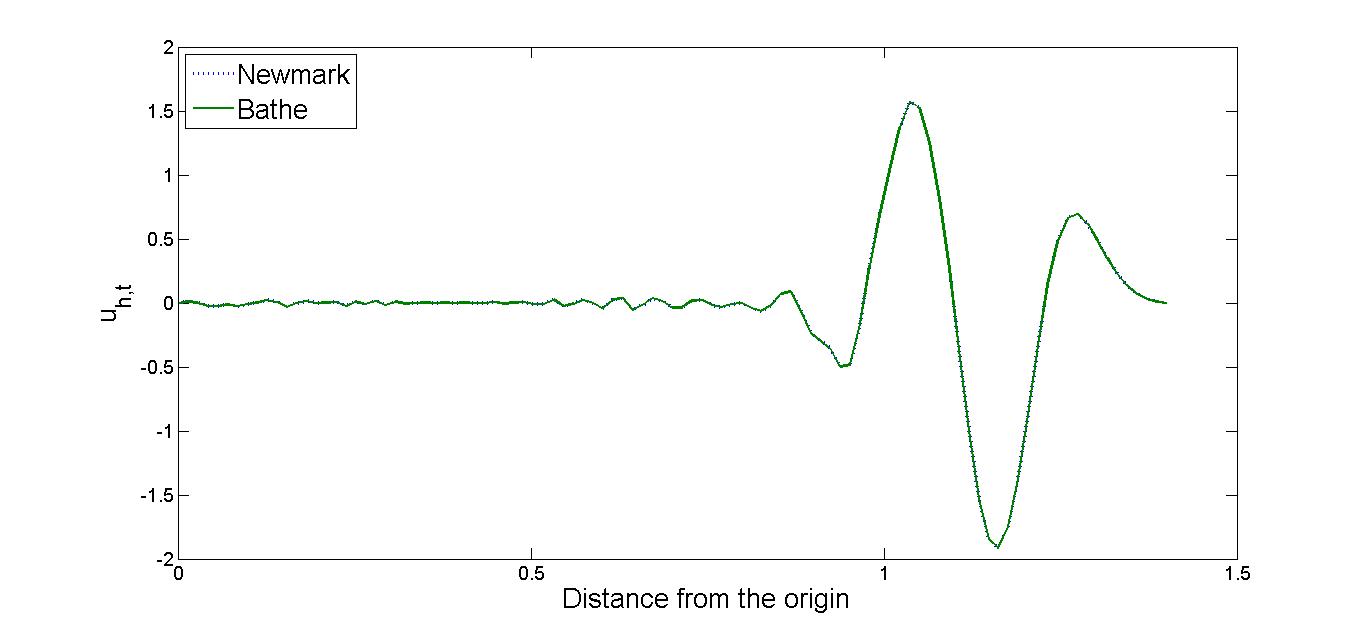} 
\caption{Displacement variations and velocity variations along the diagonal with $\tau = 1/80$.}
\label{Figure4}
}
\end{figure}
Figure \ref{Figure5} (\ref{Figure6}) shows the snapshots of the solution variable u  (the time derivative $u_t$) calculated
using the Bathe method and the Newmark method for $\tau = 1/20$.
\begin{figure}[!h]
\center{
\includegraphics[scale=0.20, draft=false]{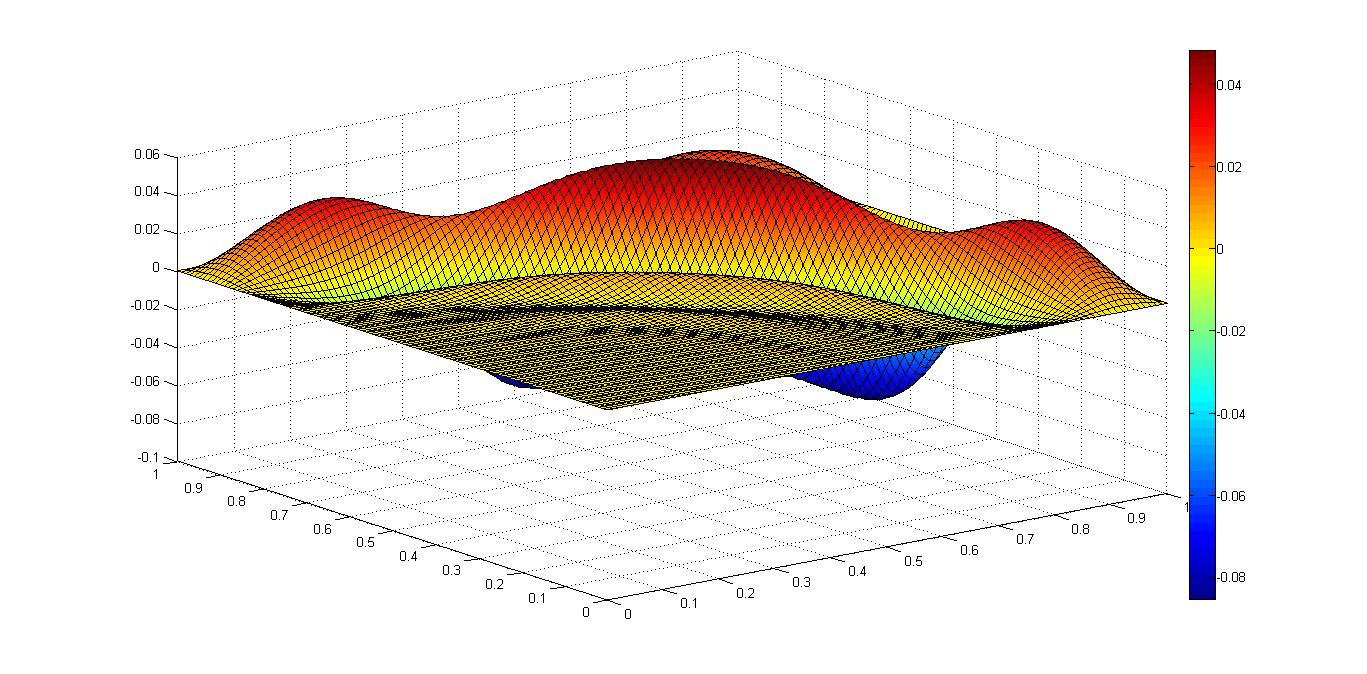} \quad
\includegraphics[scale=0.20, draft=false]{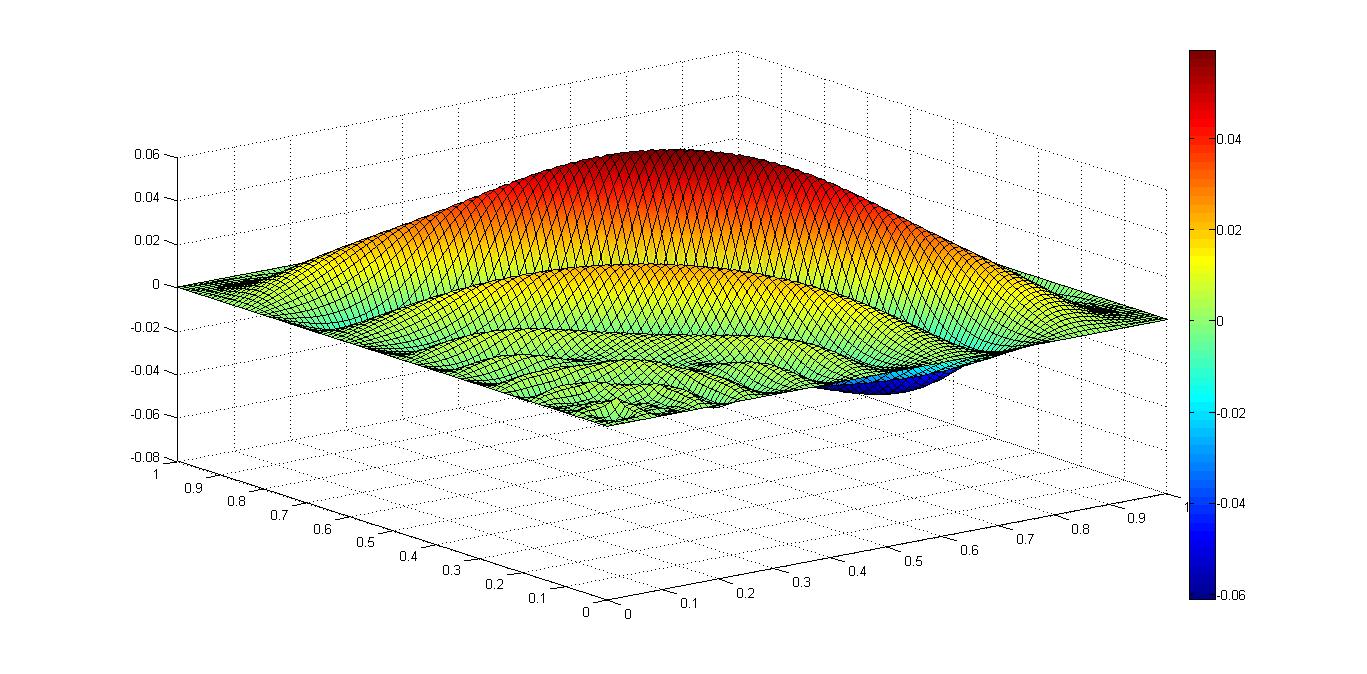}
\caption{Snapshots of displacements computed with the Bathe method and the Newmark method.}
\label{Figure5}
}
\end{figure}
\begin{figure}[!h]
\center{
\includegraphics[scale=0.20, draft=false]{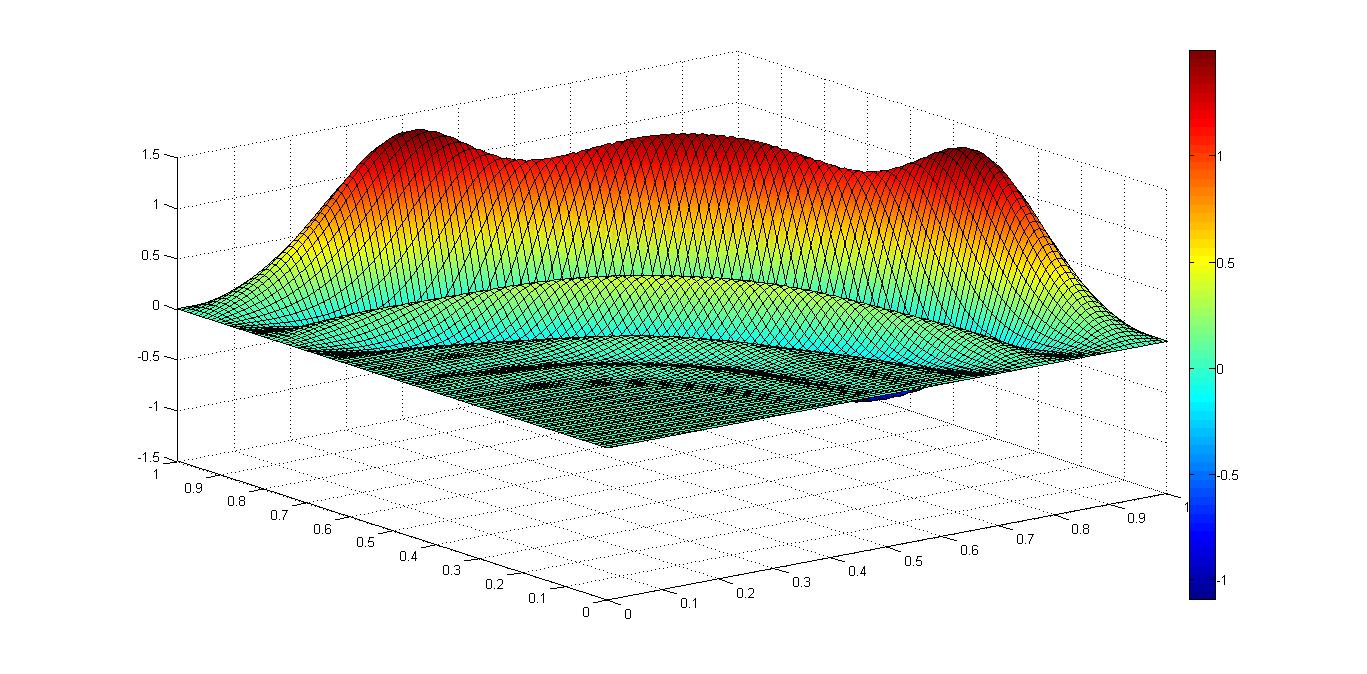} \quad
\includegraphics[scale=0.20, draft=false]{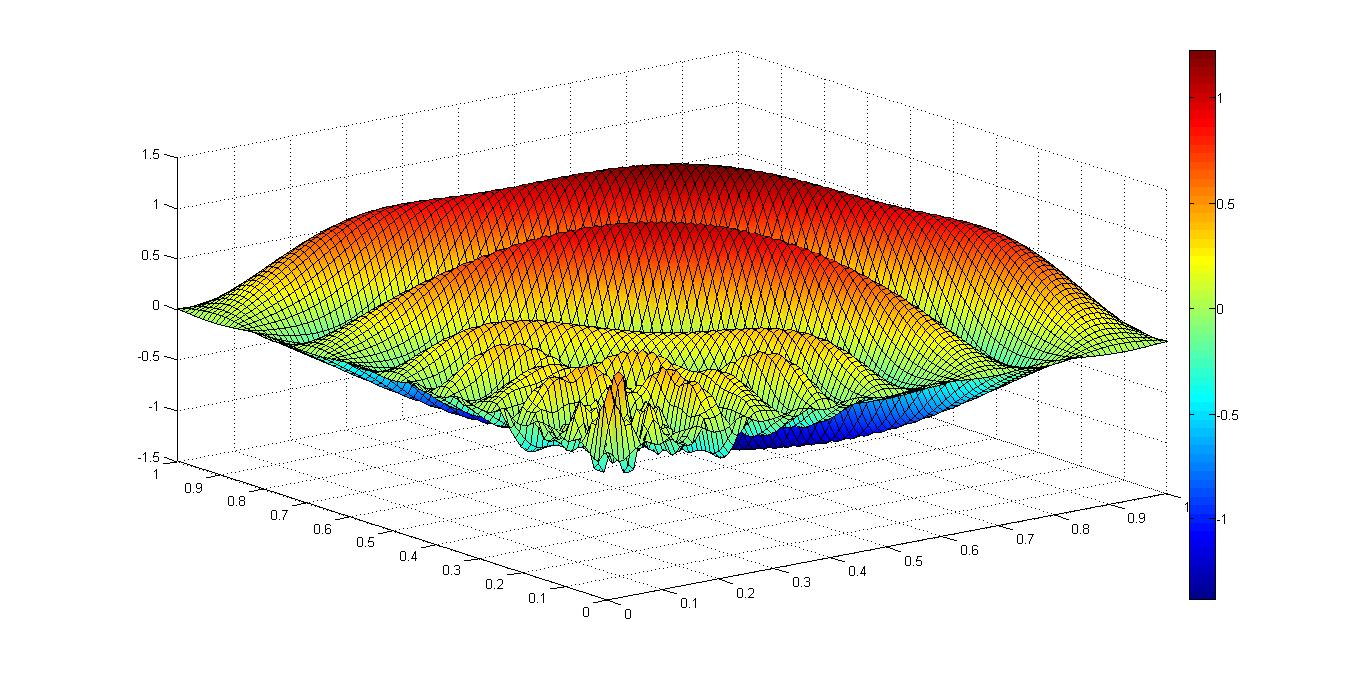}
\caption{Snapshots of velocities computed with the Bathe method and the Newmark method.}
\label{Figure6}
}
\end{figure}
We can observe that the Bathe method significantly improves the accuracy of the solution and is not is affected by spurious oscillations also for small values of CLF.
\end{test}

\section{Acknowledgements}
The author wishes to thank Luciano Lopez for several interesting discussions and suggestions on the paper.

The author wishes to thank the National Group of Scientific Computing (GNCS-INDAM) that through the project ``Finanziamento Giovani Ricercatori 2015-2016'' has supported this research.

\addcontentsline{toc}{section}{\refname}
\bibliographystyle{plain}
\bibliography{bibliografia}

\begin{thebibliography}{10}

\bibitem{VEM-enhanced}
B.~Ahmad, A.~Alsaedi, F.~Brezzi, L.~D. Marini, and A.~Russo.
\newblock Equivalent projectors for virtual element methods.
\newblock {\em Comput. Math. Appl.}, 66(3):376--391, 2013.

\bibitem{VEM-stream}
P.~F. Antonietti, L.~Beir{\~a}o~da Veiga, D.~Mora, and M.~Verani.
\newblock A stream virtual element formulation of the {S}tokes problem on
  polygonal meshes.
\newblock {\em SIAM J. Numer. Anal.}, 52(1):386--404, 2014.

\bibitem{VEM-cahn}
P.~F. Antonietti, L.~Beir{\~a}o Da~Veiga, S.~Scacchi, and MArco Verani.
\newblock A $c^1$ {V}irtual {E}lement method for the {C}ahn-{H}illiard equation
  with polygonal meshes.
\newblock {\em arXiv preprint arXiv:1502.03259}, 2015.

\bibitem{VEM-nonconforming}
B.~Ayuso~de Dios, K.~Lipnikov, and G.~Manzini.
\newblock The nonconforming virtual element method.
\newblock {\em arXiv preprint arXiv:1405.3741}, 2015.

\bibitem{bathe1}
K.-J. Bathe.
\newblock Conserving energy and momentum in nonlinear dynamics: a simple
  implicit time integration scheme.
\newblock {\em Comput. \& Structures}, 85(7-8):437--445, 2007.

\bibitem{bathe4}
K.J. Bathe, F.~Brezzi, and L.D. Marini.
\newblock The {MITC}9 {\it shell} element in {\it plate} bending: mathematical
  analysis of a simplified case.
\newblock {\em Comput. Mech.}, 47(6):617--626, 2011.

\bibitem{VEM-volley}
L.~Beir{\~a}o~da Veiga, F.~Brezzi, A.~Cangiani, G.~Manzini, L.~D. Marini, and
  A.~Russo.
\newblock Basic principles of virtual element methods.
\newblock {\em Math. Models Methods Appl. Sci.}, 23(1):199--214, 2013.

\bibitem{VEM-elasticity}
L.~Beir{\~a}o~da Veiga, F.~Brezzi, and L.~D. Marini.
\newblock Virtual elements for linear elasticity problems.
\newblock {\em SIAM J. Numer. Anal.}, 51(2):794--812, 2013.

\bibitem{VEM-hitchhikers}
L.~Beir{\~a}o~da Veiga, F.~Brezzi, L.~D. Marini, and A.~Russo.
\newblock The {H}itchhiker's {G}uide to the {V}irtual {E}lement {M}ethod.
\newblock {\em Math. Models Methods Appl. Sci.}, 24(8):1541--1573, 2014.

\bibitem{VEM-mixedgeneral}
L.~Beir{\~a}o Da~Veiga, F.~Brezzi, L.~D. Marini, and A.~Russo.
\newblock Mixed {V}irtual {E}lement {M}ethods for general second order elliptic
  problems on polygonal meshes.
\newblock {\em arXiv preprint arXiv:1506.07328}, 2015.

\bibitem{VEM-serendipity}
L.~Beir{\~a}o Da~Veiga, F.~Brezzi, L.~D. Marini, and A.~Russo.
\newblock Serendipity {N}odal {VEM} spaces.
\newblock {\em arXiv preprint arXiv:1510.08477}, 2015.

\bibitem{VEM-general}
L.~Beir{\~a}o Da~Veiga, F.~Brezzi, L.~D. Marini, and A.~Russo.
\newblock Virtual element methods for general second order elliptic problems on
  polygonal meshes.
\newblock {\em arXiv preprint arXiv:1412.2646}, 2015.

\bibitem{VEM-conforming}
L.~Beir{\~a}o Da~Veiga, F.~Brezzi, L.D. Marini, and A.~Russo.
\newblock ${H}(div)$ and ${H}(curl)$-conforming {VEM}.
\newblock {\em Numer. Math.}, pages 1--30, 2015.

\bibitem{BLM11book}
L.~Beir{\~a}o~da Veiga, K.~Lipnikov, and G.~Manzini.
\newblock {\em The mimetic finite difference method for elliptic problems},
  volume~11 of {\em MS\&A. Modeling, Simulation and Applications}.
\newblock Springer, 2014.

\bibitem{mfdpreprint}
L.~Beir{\~a}o Da~Veiga, L.~Lopez, and G.~Vacca.
\newblock Symplectic {M}imetic {F}inite {D}ifference methods for {H}amiltonian
  wave equations in 2{D}.
\newblock {\em arXiv:1505.01017}, 2015.

\bibitem{BLM15}
L.~Beir{\~a}o~da Veiga, C.~Lovadina, and D.~Mora.
\newblock A {V}irtual {E}lement {M}ethod for elastic and inelastic problems on
  polytope meshes.
\newblock {\em Comput. Methods Appl. Mech. Engrg.}, 295:327--346, 2015.

\bibitem{VEM-preprint}
L.~Beir{\~a}o Da~Veiga, C.~Lovadina, and G.~Vacca.
\newblock Divergence free virtual elements for the {S}tokes problem on
  polygonal meshes.
\newblock {\em arXiv preprint arXiv: 1510.01655}, 2015.

\bibitem{berrone}
M.~F. Benedetto, S.~Berrone, S.~Pieraccini, and S.~Scial{\`o}.
\newblock The virtual element method for discrete fracture network simulations.
\newblock {\em Comput. Methods Appl. Mech. Engrg.}, 280:135--156, 2014.

\bibitem{scott}
S.~C. Brenner and L.~R. Scott.
\newblock {\em The mathematical theory of finite element methods}, volume~15 of
  {\em Texts in Applied Mathematics}.
\newblock Springer, New York, third edition, 2008.

\bibitem{VEM-mixed}
F.~Brezzi, R.~S. Falk, and L.~D. Marini.
\newblock Basic principles of mixed virtual element methods.
\newblock {\em ESAIM Math. Model. Numer. Anal.}, 48(4):1227--1240, 2014.

\bibitem{Brezzi:Marini:plates}
F.~Brezzi and L.~D. Marini.
\newblock Virtual element methods for plate bending problems.
\newblock {\em Comput. Methods Appl. Mech. Engrg.}, 253:455--462, 2013.

\bibitem{VEM-nonconforming1}
A.~Cangiani, G.~Manzini, and O.~J. Sutton.
\newblock Conforming and nonconforming virtual element methods for elliptic
  problems.
\newblock {\em arXiv preprint arXiv:1507.03543}, 2015.

\bibitem{paulinopost}
Arun~L. Gain, Cameron Talischi, and Glaucio~H. Paulino.
\newblock On the virtual element method for three-dimensional linear elasticity
  problems on arbitrary polyhedral meshes.
\newblock {\em Comput. Methods Appl. Mech. Engrg.}, 282:132--160, 2014.

\bibitem{bathe2}
S.~Ham and K.J. Bathe.
\newblock A finite element method enriched for wave propagation problems.
\newblock {\em Comput. \& Structures}, 94:1--12, 2012.

\bibitem{lopezvacca}
L.~Lopez and G.~Vacca.
\newblock Spectral properties and conservation laws in {M}imetic {F}inite
  {D}ifference methods for {PDE}s.
\newblock {\em J. Comput. Appl. Math.}, 292:760--784, 2016.

\bibitem{mora2015virtual}
D.~Mora, G.~Rivera, and R.~Rodr{\'\i}guez.
\newblock A virtual element method for the steklov eigenvalue problem.
\newblock {\em Math. Models Methods Appl. Sci.}, 25(08):1421--1445, 2015.

\bibitem{newmark}
N.~M. Newmark.
\newblock A method of computation for structural dynamics.
\newblock {\em J ENG MECH DIV}, 85(3):67--94, 1959.

\bibitem{bathe3}
G.~Noh, S.~Ham, and K.J. Bathe.
\newblock Performance of an implicit time integration scheme in the analysis of
  wave propagations.
\newblock {\em Comput. \& Structures}, 123:93--105, 2013.

\bibitem{VEM-helmholtz}
I.~Perugia, P.~Pietra, and A.~Russo.
\newblock A {P}lane {W}ave {V}irtual {E}lement {M}ethod for the {H}elmholtz
  {P}roblem.
\newblock {\em arXiv preprint arXiv:1505.04965}, 2015.

\bibitem{raviart}
P.-A. Raviart and J.-M. Thomas.
\newblock {\em Introduction \`a l'analyse num\'erique des \'equations aux
  d\'eriv\'ees partielles}.
\newblock Collection Math\'ematiques Appliqu\'ees pour la Ma\^\i trise.
  [Collection of Applied Mathematics for the Master's Degree]. Masson, Paris,
  1983.

\bibitem{TPPM12}
C.~Talischi, G.~H. Paulino, A.~Pereira, and I.~F.M~. Menezes.
\newblock Polymesher: a general-purpose mesh generator for polygonal elements
  written in matlab.
\newblock {\em Struct. Multidisc Optimiz.}, 45(3):309--328, 2012.

\bibitem{vaccabeirao}
G.~Vacca and L.~Beir{\~a}o Da~Veiga.
\newblock Virtual element methods for parabolic problems on polygonal meshes.
\newblock {\em Numer. Methods Partial Differential Equations},
  31(6):2110--2134, 2015.

\bibitem{assessment}
Y.C. Wang, V.~Murti, and S.~Valliappan.
\newblock Assessment of the accuracy of the newmark method in transient
  analysis of wave propagation problems.
\newblock {\em Earthquake eng. \& struct.}, 21(11):987--1004, 1992.

\end{thebibliography}

\end{document}